\documentclass[12pt]{article}

\usepackage[margin=1.25in]{geometry}

\usepackage{amssymb,amsmath}
\usepackage{amsthm}
\usepackage{hyperref}
\usepackage{mathrsfs}
\usepackage{textcomp}
\usepackage{enumerate}
\usepackage{framed}
\usepackage{ulem}
\usepackage{color,xcolor}

\usepackage{soul}
\soulregister\cite7 
\soulregister\citep7 
\soulregister\citet7 
\soulregister\ref7 
\soulregister\pageref7 

\hypersetup{
    colorlinks,%
    citecolor=blue,%
    filecolor=blue,%
    linkcolor=blue,%
    urlcolor=black
}

\newtheorem{theorem}{Theorem}[section]
\newtheorem{definition}{Definition}[section]
\newtheorem{lemma}{Lemma}[section]
\newtheorem{remark}{Remark}[section]
\newtheorem{proposition}{Proposition}[section]
\newtheorem{corollary}{Corollary}[section]

\numberwithin{equation}{section}

\allowdisplaybreaks[3]
\makeatletter

\newdimen\bibspace
\renewenvironment{thebibliography}[1]{%
 \section*{\refname 
       \@mkboth{\MakeUppercase\refname}{\MakeUppercase\refname}}%
     \list{\@biblabel{\@arabic\c@enumiv}}%
          {\settowidth\labelwidth{\@biblabel{#1}}%
           \leftmargin\labelwidth
           \advance\leftmargin\labelsep
           \itemsep\bibspace
           \parsep\z@skip     %
           \@openbib@code
           \usecounter{enumiv}%
           \let\p@enumiv\@empty
           \renewcommand\theenumiv{\@arabic\c@enumiv}}%
     \sloppy\clubpenalty4000\widowpenalty4000%
     \sfcode`\.\@m}
    {\def\@noitemerr
      {\@latex@warning{Empty `thebibliography' environment}}%
     \endlist}

\makeatother

           \newcommand{\ud}{\mathrm{d}}
\newcommand{\be}{\begin{equation}}      \newcommand{\ee}{\end{equation}}

\newcommand{\alp}{\ensuremath{\alpha}}
\newcommand{\bt}{\ensuremath{\beta}}
\newcommand{\wdt}{\ensuremath{\widetilde}}

\begin{document}

\title{On a Fractional Nirenberg problem involving the square root
of the Laplacian on $\mathbb{S}^{3}$}

\author{ {\sc Yan Li}\,, {\sc Zhongwei Tang}\thanks{The research was supported by National Science Foundation of China(12071036,12126306)}\, and {\sc Ning Zhou} \\
\small School of Mathematical Sciences, \\
\small Laboratory of Mathematics and Complex Systems, MOE,\\
\small Beijing Normal University, Beijing, 100875, P.R. of China}

\date{}

\maketitle

\begin{abstract}
In  this paper, we are devoted to establishing the compactness and existence results of the solutions to the fractional
Nirenberg problem  for $n=3,$ $\sigma=1/2,$ when the prescribing $\sigma$-curvature function satisfies the $(n-2\sigma)$-flatness condition near its critical points. The compactness results are new and optimal. In addition, we obtain a degree-counting formula of all solutions. From our results, we can know where blow up occur. Moreover, for any finite distinct points, the sequence of solutions that blow up precisely at these points can be constructed.
We extend the results of   Li in \cite[CPAM, 1996]{LYY} from  the local problem to nonlocal cases.
\end{abstract}


\section{Introduction}
Great attention has been focused on the study of fractional and nonlocal operators of elliptic type, both for pure mathematical research and in view of concrete real-world applications. This type of operator arises in a quite natural way in many different contexts, such as,  the thin obstacle problem, optimization,  phase transitions,
minimal surfaces, materials science, water waves,  population dynamics, geophysical fluid dynamics
and mathematical finance. For more details and applications, see \cite{app,cla,CS,cp,HMOW,JLX,JLX2,ae,jv} and references therein.

In this paper, we are concerned with the Nirenberg's problem in the
fractional setting which constitutes in itself a branch in geometric analysis.
We first introduce  the Nirenberg problem.
Let $(\mathbb{S}^n,g_{0})$ be the standard $n$-sphere.
The Nirenberg problem is the following:
 which  function $K$ on $\mathbb{S}^{2}$ is the Gauss curvature of a metric $g$ on $\mathbb{S}^2$
 conformally equivalent to $g_0?$
  If we write $g=e^{v} g_{0},$  this problem is equivalent to finding a function $v$ on
$\mathbb{S}^2$ to solving
\begin{align}\label{1.69}
-\Delta_{g_{0}}v+1=K(x)e^{2v}\quad \text{on}\,\,\mathbb{S}^2,
\end{align}
where $\Delta_{g_{0}}$ denotes the Laplace-Beltrami operator associated with the metric $g_{0}.$

Naturally one may ask a similar question in higher dimensional case, namely
which function $K$ on $\mathbb{S}^n$ $(n\geq 3)$ is the scalar curvature of a metric $g$ on $\mathbb{S}^n$
conformally equivalent to $g_{0}?$
If we write $g=v^{4/(n-2)}g_{0},$
 this problem is  equivalent to finding a function
 $v$ on $\mathbb{S}^{n}$ which satisfies the following equation:
\begin{align}\label{1.70}
-\Delta_{g_{0}} v+c(n) R_{0} v=
c(n) K(x) v^{\frac{n+2}{n-2}} \quad  \text { on } \,\mathbb{S}^{n},
\end{align}
where $c(n)=(n-2)/(4(n-1)),$ $ R_{0}=n(n-1)$ is the scalar curvature of $g_{0}$.

It is well known that a  necessary condition for solving \eqref{1.69} or \eqref{1.70} is that $K$
should be positive somewhere.
Kazdan and Warner \cite{KW} obtained another necessary
condition  for the existence of solutions by  exploiting the center
dilation conformal transformations of $\mathbb{S}^n.$

The first significant result on the Nirenberg problem was made by
Koutroufiotis \cite{Kou}, which established  the existence of the solutions to \eqref{1.69}
by assuming that $K$ is an antipodally symmetric function which close to 1.
Morse \cite{Mos} proved the existence of antipodally symmetric solutions to \eqref{1.69}  for all antipodally symmetric functions $K$ which are positive somewhere.
Later on, Chang and Yang \cite{CY}  further extended this existence result to the case of
$K$ without any symmetry assumption. In addition, Bahri and Coron \cite{BC}  gave
a sufficient condition for existence of the solutions to  \eqref{1.70} in  dimension $n=3$ by assuming
 that $K(x)$ has only  nondegenerate critical points.
As for  the compactness of all solutions in dimensions $n=2,3,$
 Chang et al. \cite{CGY}, Han \cite{Han}, and
Schoen and Zhang \cite{SZ} proved that  a sequence of solutions cannot blow up at
more than one point.

Li \cite{LYYJ, LYY}  established the compactness and existence results for \eqref{1.70}
by characterizing the flatness order of $K(x)$ near its critical points with $(*)_{\bt}$ conditions.
More precisely, the cases of $\beta>n-2$ and $\bt= n-2$ are given in
\cite{LYYJ} and \cite{LYY},  respectively.
In these two papers, the compactness result is very different from the previous low-dimensional case.
In fact, when $n=2$ or $n=3,$ a sequence of solutions to the
Nirenberg problem cannot blow up at more than one point.
However, if $n>3$, there
 could be blow up at many points, which considerably complicates the study of the problem.

The linear operators defined on left-hand  side
of \eqref{1.69} and \eqref{1.70} are called the conformal Laplacian
associated to the metric $g_{0}$ and are denoted as $P_{1}^{g_{0}}.$
For any Riemannian manifold $(M,g),$
let $R_{g}$ be the scalar curvature of $(M,g),$
and the conformal Laplacian be defined as $P_{1}^{g}=-\Delta_{g}+\frac{n-2}{4(n-1)}R_{g}.$
The Paneitz operator $P_{2}^{g}$ is
another conformal invariant operator, which was discovered by Paneitz \cite{Pa}.
Graham et al.  \cite{GJMS} generalized the
 operators $P_{1}^{g}$ and $P_{2}^{g}$ to a sequence of integer
 order conformally covariant elliptic operators $P_{k}^{g}$ for $k\in\{1,2,\cdots\}$
 if $n$ is odd, and $k\in \{ 1,\cdots, n/2\}$ if $n$ is even.
 Furthermore, Peterson \cite{Pe} constructed an intrinsically defined conformally covariant
 pseudo-differential operator of arbitrary real number order.
 Graham and Zworski \cite{GZ} introduced a mesomorphic family of conformally invariant operators on the conformal infinity of asymptotically hyperbolic manifolds.
Chang and Gonz\'alez \cite{CG} proved that the operator $P_{\sigma}^{g}$ of non-integer order $\sigma\in (0,n/2)$ can  be defined as a Dirichlet-to-Neumann operator of a conformally compact
Einstein manifold by using localization method in \cite{CS}.
This lead naturally to a fractional order curvature  $R_{\sigma}^{g}:=P_{\sigma}^{g}(1),$
which will be called  $\sigma$-curvature in this paper.
The fractional operators $P_{\sigma}^{g}$ and their associated fractional order
curvatures $P_{\sigma}^{g}(1)$  have been
the subject of many studies, for instance, see \cite{AC,Chti1,Chti2,CLZ,JLX,JLX2,Liu_cpaa}.

As in the Nirenberg problem associated to $P_{1}^{g},$
the  question of prescribing  $\sigma$-curvature  can be formulate as fractional
Nirenberg problem as follows:
which function $K$ on $\mathbb{S}^n$ is the $\sigma$-curvature of a
metric $g$ on $\mathbb{S}^n$  conformally equivalent to $g_0?$
If we denote $g=v^{4/(n-2\sigma)}g_{0},$
this problem can be expressed as finding the solution of the following
nonlinear equation with critical exponent:
\begin{align}\label{1.71}
P_{\sigma}^{g}(v)=c(n, \sigma)
K v^{\frac{n+2 \sigma}{n-2 \sigma}} \quad \text { on }\, \mathbb{S}^{n},
\end{align}
where $c(n,\sigma)=\Gamma(\frac{n}{2}+\sigma)/\Gamma(\frac{n}{2}-\sigma),$
 $K$ is a function defined on $\mathbb{S}^n,$
\begin{align*}
P_{\sigma}^{g}=\frac{\Gamma(B+\frac{1}{2}+\sigma)}{\Gamma(B+\frac{1}{2}-\sigma)}, \quad B=\sqrt{-\Delta_{g_{0}}
+\Big(\frac{n-1}{2}\Big)^{2}},
\end{align*}
 and $\Gamma$ is the Gamma function. In what follows, $P_{\sigma}^{g}$ is simply written  as $P_{\sigma}.$

Let $ K\in C^{1,1}(\mathbb{S}^n)$ be a positive function and $\beta $ is a positive constant,
we say that $K$ satisfies the flatness condition  $(*)_{\beta}$ if for every critical point $\xi_{0}$ of $K,$ in some geodesic normal coordinates $\{y_{1}, \cdots, y_{n}\}$ centered at $\xi_{0}$, there exists a small neighborhood $\mathscr{O}$ of $0$  and
$a_{j}(\xi_{0})\ne 0,$ $\sum_{j=1}^{n}a_j(\xi_{0})\neq 0,$
such that
\begin{align*}
K(y)=K(0)+\sum_{j=1}^{n} a_{j}(\xi_{0})|y_{j}|^{\beta}+R(y) \quad \text { in } \mathscr{O},
\end{align*}
where
$$\sum_{s=0}^{[\beta]}
|\nabla^s R(y)||y|^{-\bt+s}\rightarrow 0\quad \hbox{\,as\,}\,y\rightarrow 0,$$
here $\nabla^{s}$ denotes all possible derivatives of order $s$ and
$[\beta]$ is the integer part of $\beta.$

For $0<\sigma<1,$ Jin et al. \cite{JLX, JLX2} proved the existence of the solutions to \eqref{1.71}
and derived some compactness properties  when $K$ satisfies the $(*)_{\beta}$
condition with the flatness order $ \beta\in(n-2\sigma, n),$
by using the approach based on approximation of the solutions to \eqref{1.71} by a blow up
subcritical method. Since their conclusions is valid only when the flatness order
$\beta>n-2\sigma,$ some very interesting functions $K$ are excluded.
In fact, note that an important class of functions,
which is worth including in the results of existence and compactness for \eqref{1.71},
are the Morse functions with only nondegenerate critical points.
Such functions  satisfy the $(*)_{2}$ condition.

By using a self-contained approach,  the description of lack of compactness
and the existence results of the  solutions to \eqref{1.71} were given
by Abdelhedi et al. \cite{Chti1} when $\beta\in (1,n-2\sigma]$, and by Chtioui  and Abdelhedi  \cite{Chti2} when $\beta\in [n-2\sigma, n).$
However, under the  assumption of the flatness order $\beta=n-2\sigma,$
which is called the critical flatness condition in this paper,
 the precise  compactness result and the degree-counting formula  of the solutions to \eqref{1.71} is unknown.
Therefore, it is natural to  study the compactness results
when the prescribing curvature function $K$ satisfy  the critical flatness condition.
When $\sigma=1$ and  $K$ satisfy  the critical flatness condition, namely $\beta=n-2, $ the compactness and existence results of the
solutions to \eqref{1.70} was obtained by Li \cite{LYY}.

What we  consider here is  the case when the  prescribing $\sigma$-curvature function  satisfy the critical  flatness order $\beta=n-2\sigma=2,$   which include an important  class of  functions,  for instance the Morse functions.
In addition, we can establish the optimal compactness result and give a degree-counting formula of all solutions to \eqref{1.71} in this case.
In this paper, we study  the equation \eqref{1.71}  when  $n=3$ and $\sigma=1/2.$
Especially, from our results, we show that a sequence of solutions to
\eqref{1.71} can blow up at more than one point and   for any finite distinct points
on $\mathbb{S}^3,$  we can construct a sequence of solutions to \eqref{1.71} that blow up precisely at these points.
In a forthcoming paper, we deal with the higher order case, i.e.,
$n=2\sigma+2$ for any $1<\sigma<n/2.$

Before state our results, we introduce some definitions and notations.

For $\sigma\in (0,1),$ the fractional Laplacian  is a
nonlocal pseudo-differential operator, taking the form:
\begin{align}
\begin{aligned}\label{2.7}
(-\Delta)^{\sigma} u(x):
&=
C(n, \sigma)P.\,V.\, \int_{\mathbb{R}^{n}} \frac{u(x)-u(y)}{|x-y|^{n+2 \sigma}} \,\ud y\\
&=C(n, \sigma) \lim _{\varepsilon \rightarrow 0^{+}} \int_{\mathbb{R}^{n} \backslash B_{\varepsilon}(x)} \frac{u(x)-u(y)}{|x-y|^{n+2 \sigma}} \,\ud y, \quad x \in \mathbb{R}^{n},
\end{aligned}
\end{align}
where $B_{\varepsilon}(x)$ is the ball centered at $x \in \mathbb{R}^{n}$
with radius $\varepsilon.$
Here $P.\,V.\,$ is a commonly used abbreviation for ``in the principal value sense'' and $C(n, \sigma)$ is a dimensional constant that depends on $n$ and $\sigma$, precisely given by
$$
C(n, \sigma):=\Big(\int_{\mathbb{R}^{n}}
\frac{1-\cos \zeta_{1}}{|\zeta|^{n+2 \sigma}} \ud \zeta\Big)^{-1}
$$
with $\zeta=(\zeta_{1}, \zeta^{\prime}),$ $\zeta^{\prime} \in \mathbb{R}^{n-1}.$

The singular integral given in \eqref{2.7} can be written as a weighted second-order differential quotient as follows (see \cite[Lemma 3.2]{EGE}):
$$
(-\Delta)^{\sigma} u(x):=-\frac{1}{2} C(n, \sigma)
\int_{\mathbb{R}^{n}} \frac{u(x+y)+u(x-y)-2 u(x)}{|y|^{n+2 \sigma}} \,\ud y, \quad x \in \mathbb{R}^{n}.$$

 This  operator is well defined in $\mathscr{S}$, the Schwartz space of rapidly decreasing $C^{\infty}$ function in $\mathbb{R}^{n}$, and it can be equivalently defined in terms of the Fourier transform:
$$
(-\Delta)^{\sigma} u(x):=\mathscr{F}^{-1}
(|\xi|^{2 \sigma}(\mathscr{F} u)(\xi))(x), \quad x \in \mathbb{R}^{n}.
$$
where  $\mathscr{F}$ denotes the Fourier transform operator.

Let $\dot{H}^{\sigma}(\mathbb{R}^{n})$ denote the closure of the set
$C_{c}^{\infty}(\mathbb{R}^{n})$ of compactly supported smooth functions under the norm
$$
\|u\|_{\dot{H}^{\sigma}(\mathbb{R}^{n})}=\||\xi|^{\sigma} \mathscr{F}(u)(\xi) \|_{L^{2}(\mathbb{R}^{n})}.
$$
For any $u \in \dot{H}^{\sigma}\left(\mathbb{R}^{n}\right)$, we set
\begin{align}\label{2.8}
U(x, t)=\mathcal{P}_{\sigma}[u]:=\int_{\mathbb{R}^{n}} \mathcal{P}_{\sigma}(x-\xi, t) u(\xi) \ud \xi, \quad(x, t) \in \mathbb{R}_{+}^{n+1}:=\mathbb{R}^{n} \times(0, \infty),
\end{align}
where
$$
\mathcal{P}_{\sigma}(x, t)=\beta(n, \sigma) \frac{t^{2 \sigma}}{(|x|^{2}+t^{2})^{(n+2 \sigma) / 2}},
$$
with a constant $\beta(n, \sigma)$ such that $\int_{\mathbb{R}^{n}} \mathcal{P}_{\sigma}(x, 1) \ud x=1$.   Let us denote that
for any  open set  $D\subset\mathbb{R}_+^{n+1}$, the space  $L^{2}(t^{1-2 \sigma}, D)$  is the Banach space endowed with the norm
$$
\|V\|_{L^{2}(t^{1-2 \sigma}, D)}:=\Big(\int_{D}t^{1-2 \sigma}
 V^{2} \,\ud X\Big)^{1 / 2}<\infty,
$$
for any $V\in L^{2}(t^{1-2 \sigma}, D)$.
Then the above $U(x,t) \in L^{2}(t^{1-2 \sigma}, K)$ for any compact set $K$ in $\overline{\mathbb{R}_{+}^{n+1}}, \nabla U(x,t)  \in L^{2}(t^{1-2 \sigma}$, $\mathbb{R}_{+}^{n+1})$ and $U(x,t)  \in C^{\infty}(\mathbb{R}_{+}^{n+1}).$

By the celebrated work by Caffarelli and Silvestre (see \cite{CS}), one can find that $U(x,t)$ satisfies
\begin{align*}
\operatorname{div}(t^{1-2 \sigma} \nabla U)=0 \quad \text { in } \mathbb{R}_{+}^{n+1},
\end{align*}
\begin{align*}
\|\nabla U\|_{L^{2}(t^{1-2 \sigma}, \mathbb{R}_{+}^{n+1})}=N_{\sigma}\|u\|_{\dot{H}^{\sigma}(\mathbb{R}^{n})},
\end{align*}
and
\begin{align*}
-\lim _{t \rightarrow 0} t^{1-2 \sigma} \partial_{t} U(x, t)=N_{\sigma}(-\Delta)^{\sigma} u(x) \quad \text { in } \mathbb{R}^{n}
\end{align*}
in the distribution sense, where $N_{\sigma}=2^{1-2\sigma}\Gamma(1-\sigma)/\Gamma(\sigma).$
Here  one refer $U(x,t)=\mathcal{P}_{\sigma}[u]$ in \eqref{2.8}  as the extension of $u \in \dot{H}^{\sigma}(\mathbb{R}^{n})$.

Let $\Omega \subset \mathbb{R}^{n}$ $(n \geq 3)$ be a domain,
$\tau_{i} \geq 0,$ $i=1,2,\cdots,$ satisfy $\lim _{i \rightarrow \infty} \tau_{i}=0,$ $p_{i}=
(n+2 \sigma) /(n-2 \sigma)-\tau_{i}$, and $K_{i} \in C^{1,1}(\Omega)$ satisfy, for some constants $A_{1},$ $A_{2}>0,$
\begin{align}\label{1.49}
1 / A_{1} \leq K_{i}(x) \leq A_{1} \quad \text { for all } x \in \Omega, \quad\left\|K_{i}\right\|_{C^{1,1}(\Omega)} \leq A_{2}.
\end{align}
Let $u_{i} \in L^{\infty}(\Omega) \cap \dot{H}^{\sigma}(\mathbb{R}^{n})$ with $u_{i} \geq 0$ in $\mathbb{R}^{n}$ satisfy
\begin{align}\label{1.50}
(-\Delta)^{\sigma} u_{i}=c(n, \sigma) K_{i} u_{i}^{p_{i}} \quad \text { in }\, \Omega,
\end{align}
where $c(n,\sigma)$ is as in \eqref{1.71}.

\begin{definition}
  Suppose that $\{K_{i}\}$ satisfies \eqref{1.49} and $\{u_{i}\}$ satisfies \eqref{1.50}.
  A point $\overline{y} \in \Omega$ is called a blow up point of $\{u_{i}\}$ if there exists a sequence
  $y_{i}$ tending to $\overline{y}$ such that $u_{i}(y_{i}) \rightarrow \infty$.
\end{definition}

\begin{definition}\label{defn1.1}
A blow up point $\overline{y} \in \Omega$ is called an isolated blow up point of $\{u_{i}\}$ if there exist $0<\overline{r}<\operatorname{dist}(\overline{y}, \Omega)$, $\overline{C}>0$, and $a$
sequence $y_{i}$ tending to $\bar{y}$, such that $y_{i}$ is a local maximum point of $u_{i},$ $u_{i}(y_{i}) \rightarrow \infty$ and
\begin{align}\label{1.57}
u_{i}(y) \leq \overline{C}|y-y_{i}|^{-2 \sigma (p_{i}-1)} \quad \text { for all }\, y \in B_{\overline{r}}(y_{i}).
\end{align}
\end{definition}

Let $y_{i} \rightarrow \overline{y}$ be an isolated blow up point of $\{u_{i}\}$, and define,
for $r>0$,
$$
\overline{u}_{i}(r):=\frac{1}{|\partial B_{r}(y_{i})|} \int_{\partial B_{r}(y_{i})} u_{i}
\quad \text{and}\quad
\overline{w}_{i}(r):=r^{2 \sigma /(p_{i}-1)} \overline{u}_{i}(r).
$$
\begin{definition}
A point  $y_{i} \rightarrow \overline{y} \in \Omega$ is called an isolated simple blow up point if $y_{i} \rightarrow \overline{y}$ is an isolated blow up point such that for some $\rho>0$ (independent of i), $\overline{w}_{i}$ has precisely one critical point in $(0, \rho)$ for large $i$.
\end{definition}

For $K \in C^{2}(\mathbb{S}^{3})$, we introduce the following notation:
\begin{align}\label{1.65}
\begin{aligned}
\mathscr{K}&=\{q \in \mathbb{S}^{3}: \nabla_{g_{0}} K(q)=0\}, \\
\mathscr{K}^{+}&=\{q \in \mathbb{S}^{3}: \nabla_{g_0} K(q)=0,\,
\Delta_{g_0} K(q)>0\}, \\
\mathscr{K}^{-}&=\{q \in \mathbb{S}^{3}: \nabla_{g_{0}} K(q)=0,\, \Delta_{g_{0}} K(q)<0\},\\ \mathscr{M}_{K}&=\{v \in C^{2}(\mathbb{S}^{3}): v \text { satisfies \eqref{1.71}} \}.
\end{aligned}
\end{align}
 For any $k\in\mathbb{N}_{+}$ distinct points $q^{(1)}, \cdots, q^{(k)}
\in \mathscr{K} \backslash \mathscr{K}^{+},$
the following $k \times k$ real symmetric matrix $M$ is defined by,
for $i,j=1,\cdots,k,$
\begin{align}
\begin{aligned}\label{M}
&M_{ii}=-
 \frac{\Delta_{g_{0}}K(q^{(i)})}
{K(q^{(i)})^{3}},\\
&M_{i j}
=-6\frac{G_{q^{(i)}}(q^{(j)})}
{K(q^{(i)}) K(q^{(j)})},\quad i\ne j,
\end{aligned}
\end{align}
where
\begin{align}\label{gf}
G_{q^{(i)}}(q^{(j)})
=&\frac{1}{1-\cos d(q^{(i)},q^{(j)})}
\end{align}
 is the Green's function of $P_{\sigma}$ on $\mathbb{S}^3,$
 and  $d(\cdot\,,\,\cdot)$ denotes the geodesic distance.
Let $\mu(M)$ denote the smallest eigenvalue of $M$, and
when $k=1,$
 $$
 \mu(M)=M=-\frac{\Delta_{g_0}K(q^{(1)})}{K(q^{(1)})^3}.$$

Now we are going to present our first  result about characterization
of blow up behavior of the solutions, which is:
\begin{theorem}\label{thm1}
Let $K \in C^{2}\left(\mathbb{S}^{3}\right)$ be a positive function and $\mathscr{K},
\mathscr{K}^{-},\mathscr{K}^{+}$ be as in \eqref{1.65}.
Let $p_{i}$ satisfy $p_{i} \leq 2,$ $ p_{i} \rightarrow 2,$ $\tau_{i}=2-p_{i},$
$K_{i} \in C^{2}(\mathbb{S}^{3})$ satisfy $K_{i} \rightarrow K$
in $C^{2}(\mathbb{S}^{3}),$ and $v_{i}\in C^{2}(\mathbb{S}^3)$ satisfy
\begin{align}\label{1.1}
P_{\sigma}(v_{i})= K_{i} v_{i}^{p_i}, \quad v_{i}>0\quad \text{ on }\,\mathbb{S}^{3},
\end{align}
and
\begin{align*}
\lim_{i\rightarrow\infty}\max_{\mathbb{S}^3}v_{i}=\infty.
\end{align*}
Then there exists a constant $\delta^{*}>0$
depending only on $ \min_{\mathbb{S}^{3}} K$ and $\|K\|_{C^{2}(\mathbb{S}^{3})},$
such that after passing to a subsequence,
\begin{enumerate}
  \item[(i)]  $\{v_i\}$ (still denote the subsequence by $\{v_i\}$)  has only
  isolated simple blow up points $q^{(1)}, \cdots, q^{(k)} \in \mathscr{K} \backslash \mathscr{K}^{+}$
   $(k \geq 1)$ with $|q^{(j)}-q^{(\ell)}| \geq \delta^{*},$  $\forall\, j \neq \ell,$
  and $\mu(M(q^{(1)}, \cdots, q^{(k)})) \geq 0.$
  Furthermore, $q^{(1)}, \cdots, q^{(k)} \in \mathscr{K}^{-}$ if $k \geq 2.$
  \item[(ii)]
    Let $q^{(1)},\cdots,q^{(k)}$ be as in (i), and
     $q_{i}^{(j)}$ be the local maximum of $v_i$
     with $q_{i}^{(j)}\rightarrow q^{(j)},$ we have
  \begin{align}
  \lambda_{j}&:=K(q^{(j)})^{-1}
  \lim_{i\rightarrow\infty}v_{i}(q_{i}^{(1)})(v_{i}(q_{i}^{(j)}))^{-1}\in (0,\infty),
  \label{1.86}\\
  \mu^{(j)}&:=\lim_{i\rightarrow \infty} \tau_{i} v_{i}(q_{i}^{(j)})^{2} \in[0, \infty).\label{1.74}
 \end{align}
  \item[(iii)]
  Let $\lambda_{j},\mu^{(j)}, j=1,\cdots,k$ be as in (ii), then when $k=1,$
    \begin{align}\label{1.6}
    \mu^{(1)}=-4\frac{\Delta_{g_0} K(q^{(1)})}{K(q^{(1)})^{3}},
     \end{align}
when $k\geq 2$,
\begin{align}\label{1.10}
\sum_{\ell=1}^{k} M_{\ell j}(q^{(1)}, \cdots, q^{(k)})
\lambda_{\ell}= \lambda_{j} \mu^{(j)}, \quad \forall\, j: 1 \leq j \leq k.
\end{align}

\item [(iv)]
$\mu^{(j)}\in(0,\infty),$ $\forall\,j=1, \cdots, k,$ if and only if $\mu(M(q^{(1)},\cdots,q^{(k)})) >0.$
\end{enumerate}
\end{theorem}

In what follows, we define
\begin{align}\label{1.79}
\begin{aligned}
\mathscr{A}=\{&K\in C^{2}(\mathbb{S}^{3}):
K>0~\text{on}~\mathbb{S}^{3},\,\Delta_{g_{0}} K \neq 0 \text { on } \mathscr{K},\\
&\mu(M(q^{(1)}, \cdots, q^{(k)})) \neq 0, \,\forall\, q^{(1)},
\cdots, q^{(k)} \in \mathscr{K}^{-},\, k \geq 2\},
\end{aligned}
\end{align}
and
\begin{align}\label{1.80}
C^{2}(\mathbb{S}^{3})^{+}:=\{K \in C^{2}
(\mathbb{S}^{3}): K>0\,\, \text{on} \,\,\mathbb{S}^{3}\}.
\end{align}
  It is obvious that $\mathscr{A}$ is open in $C^{2}(\mathbb{S}^{3})$ and
   $\mathscr{A}$ is dense in $C^2(\mathbb{S}^3)^+$
with respect to the $C^{2}$ norm.

We will introduce an integer-valued continuous function Index: $\mathscr{A}\rightarrow\mathbb{Z},$
which has an explicit formula for $K\in \mathscr{A}$ being
a Morse function.
\begin{definition}\label{defn1.2}
We define $\mathrm{Index}$: $\mathscr{A}\rightarrow \mathbb{Z}$ by the following properties:
\begin{enumerate}
\item [(i)] For any Morse function $K\in \mathscr{A}$ with $\mathscr{K}^{-}=\{ q^{(1)},\cdots, q^{(s)}\},$
we define
\begin{align*}
\mathrm{Index}(K)=
-1+\sum_{k=1}^{s}
\sum_{\substack{\mu(M(q^{(i_1)},\cdots, q^{(i_{k})}))>0\\
1\leq i_{1}< \cdots< i_k \leq s}}(-1)^{k-1+\sum_{j=1}^{k}i(q^{(i_{j})})},
\end{align*}
where $i(q^{(i_{j})})$ denotes the Morse index of $K$ at $q^{(i_{j})}.$

\item [(ii)] $\mathrm{Index}:$ $\mathscr{A}\rightarrow \mathbb{Z}$ is continuous with respect to the $C^{2}(\mathbb{S}^3)$ norm of $\mathscr{A}$ and hence is locally constant.
\end{enumerate}
\end{definition}

\begin{remark}
The existence and  uniqueness of the $\mathrm{Index}$ mapping  follows
from Theorem \ref{thm2.1} and the proof of Theorem \ref{thm4} below.
\end{remark}

Our second result is about the compactness of the solutions  when $K\in \mathscr{A},$
 which is:
\begin{theorem}\label{thm2.1}
Let $\mathscr{A}$ be as in \eqref{1.79} and $K\in \mathscr{A}.$  Then
there exists a  constant $C=C(K)>0,$ such that for any $K_{i}\rightarrow K$  in
$C^{2}(\mathbb{S}^3),$ and any $v_{i}\in \mathscr{M}_{K_{i}},$ we have
\begin{align}\label{1.93}
1/C \leq \liminf _{i \rightarrow \infty}(\min_{\mathbb{S}^{3}} v_{i})
\leq \limsup _{i \rightarrow \infty}(\max _{{\mathbb{S}}^{3}} v_{i}) \leq C.
\end{align}
Furthermore,  for any $\alpha\in(0,1),$ there exists a constant $R=R(K,\alpha)>0,$ such that for any
$v\in \mathscr{M}_{K},$ we have
\begin{align*}
1/R<v(x)<R,\quad\forall\, x\in \mathbb{S}^3 \quad\hbox{and}\quad
\|v\|_{C^{2,\alpha}(\mathbb{S}^3)}<R,
\end{align*}
where  $\mathscr{M}_{K_i}$ and $\mathscr{M}_{K}$ are as in \eqref{1.65}.
\end{theorem}

For any given $0<\alpha<1,$ $R>0,$   we define
\begin{align}\label{1.82}
\mathscr{O}_{R}:=\{v\in C^{2,\alpha}(\mathbb{S}^3):
1/R<v<R,\, \|v\|_{C^{2,\alpha}(\mathbb{S}^3)}<R\}.
\end{align}
Our third  result is about  degree-counting  formula and the existence of the solutions to \eqref{1.71},
which is:
\begin{theorem}\label{thm4}
Let $\mathscr{A}$ be as in \eqref{1.79}, $K\in \mathscr{A}$
and $\mathrm{Index}(K)$ be as in Definition \ref{defn1.2}.
Then for any $\alpha\in (0,1),$ there exists a constant $R_{0}=R_{0}(K,\alpha),$ such that for all $R>R_{0},$
we have
\begin{align}\label{1.38}
\deg_{C^{2,\alpha}}(v-P_{\sigma}^{-1}(Kv^{2}), \mathscr{O}_{R}, 0)=\mathrm{Index}(K),
\end{align}
where  $\deg_{C^{2,\alpha}}$ denotes the Leray-Schauder degree in $C^{2,\alpha}(\mathbb{S}^3).$

Furthermore, if $\mathrm{Index}(K)\ne 0,$ then
\eqref{1.71} has at least one solution.
\end{theorem}

\begin{remark}\label{re1}
It follows from Theorem \ref{thm1} that when $K\in\mathscr{A},$ the solutions to \eqref{1.71}  belong to $\mathscr{O}_{R}$ for some
$R>0.$ We call the left-hand side of \eqref{1.38}  the total degree of the solutions to the fractional equation.
 From Theorem \ref{thm4}, the total degree is $\mathrm{Index}(K).$
\end{remark}

For any finite subset $\mathcal{R}\subset \mathbb{S}^{3},$
we use $\sharp\mathcal{R}$ to denote the number of elements in the set $\mathcal{R}$.
Let us now state a corollary of Theorem \ref{thm4}, which is:
\begin{corollary}\label{cor1}
Let $\mathscr{A}$ be as in \eqref{1.79} and $K\in \mathscr{A}$ be a
 Morse function satisfying $\sharp\mathscr{K}^{-}\leq 1$ or
for any distinct $P,$ $Q\in \mathscr{K}^{-},$
\begin{align}\label{1.97}
\Delta_{g_{0}}K(P)\Delta_{g_{0}}K(Q)<9K(P)K(Q).
\end{align}
Then for any $\alpha\in(0,1),$ there exists a constant $C=C(K,\alpha)>0,$
such that for all solutions $v$ of \eqref{1.71}, we have
$$
1/C<v(x)<C,\quad \forall\, x\in\mathbb{S}^3,\quad\|v\|_{C^{2, \alpha}(\mathbb{S}^3)}<C,
$$
and for all $R\geq C,$
\begin{align*}
\deg_{C^{2,\alpha}} (v-P_{\sigma}^{-1}(Kv^2), \mathscr{O}_{R}, 0)
=\mathrm{Index} (K)=-1+\sum_{\substack{\nabla_{g_0}K(q_{0})=0
\\ \Delta_{g_{0}}K(q_0)<0}}(-1)^{i(q_{0})},
\end{align*}
where $i(q_0)$ denotes the Morse index of $K$ at $q_{0}.$

Furthermore, if
\begin{align*}
\sum_{\substack{\nabla_{g_0}K(q_{0})=0
\\ \Delta_{g_{0}}K(q_0)<0}}(-1)^{i(q_{0})}\ne 1,
\end{align*}
then \eqref{1.71} has at least one solution.
\end{corollary}

Our fourth result is about the blow up behavior of the solutions when the $\sigma$-curvature function
$K\notin \mathscr{A},$ which is:
\begin{theorem}\label{thm2}
Let $\mathscr{A}$ be as in \eqref{1.79} and $C^{2}(\mathbb{S}^3)^{+}$ be as in \eqref{1.80}.
Then for any $K \in C^{2}(\mathbb{S}^{3})^{+} \backslash \mathscr{A}=\partial \mathscr{A},$
there exists $K_{i} \rightarrow K$ in $C^{2}(\mathbb{S}^{3})$ and $v_{i}
\in \mathscr{M}_{K_{i}},$ such that
\begin{align}\label{1.72}
\lim _{i \rightarrow \infty}(\max _{\mathbb{S}^{3}} v_{i})=\infty,
\quad \lim _{i \rightarrow \infty}(\min _{\mathbb{S}^{3}} v_{i})=0,
\end{align}
where  $\mathscr{M}_{K_{i}}$ is as in \eqref{1.65}.
\end{theorem}

From Remark \ref{re1}, the total degree of solutions to \eqref{1.71} strongly depend on the sign of the  smallest eigenvalue  of $M(q^{(1)},\cdots, q^{(k)}),$
which plays a role in counting the total degree of solutions  and in the compactness result.
 In fact, the points $q^{(1)}, \cdots, q^{(k)}$ for which $\mu(M(q^{(1)}, \cdots, q^{(k)}))$ is positive characterize the so-called asymptotic in the theory of critical points at infinity developed by Bahri \cite{Bahri,
 BC}. For instance, considering a continuous family of functions $K_{t}$ $(0\leq t\leq 1),$ the total degree changes when the smallest eigenvalue of $M_{t}(q^{(1)}, \cdots, q^{(k)})$ crosses zero while it remains unchanged when other eigenvalues cross zero.

It follows from Theorem \ref{thm2} that  when $K\notin \mathscr{A},$ the solutions
to \eqref{1.71} may blow  up. A natural question is where the  blow up occur?
The following results present the accurate location of the blow up.

For any $K\in C^{2}(\mathbb{S}^3)^{+},$  we first define
\begin{align}\label{1.81}
\begin{gathered}
\mathscr{H}(K)=\left\{(q^{(1)}, \cdots, q^{(k)}):\,k \geq 1,\,q^{(j)} \in \mathscr{K} \backslash \mathscr{K}^{+},\, \forall j: 1 \leq j \leq k,\right. \\
\left.\quad q^{(j)} \neq q^{(\ell)},\, \forall j \neq \ell,\, \mu(M(q^{(1)}, \cdots, q^{(k)}))=0\right\}.
\end{gathered}
\end{align}
Our fifth result is about the location of blowing up when $K\notin \mathscr{A},$ which is:
\begin{theorem}\label{thm5}
Let $\mathscr{A}$ be as in \eqref{1.79} and $C^{2}(\mathbb{S}^3)^{+}$ be as in \eqref{1.80}.
For a given function $K\in C^{2}(\mathbb{S}^3)^+\backslash\mathscr{A},$ we have  the following results:
\begin{enumerate}
  \item[(i)]
  For any $K_{i}\rightarrow K$ in $C^{2}(\mathbb{S}^3),$
and $v_{i}\in \mathscr{M}_{K_{i}}$ with $\max_{\mathbb{S}^3} v_{i}\rightarrow \infty,$
then for some $ (q^{(1)},\cdots, q^{(k)}) \in \mathscr{H}(K),$ $\{v_{i}\}$ (after passing to
a subsequence) blows up at precisely the $k$ points.
  \item[(ii)]
 For any  $(q^{(1)},\cdots, q^{(k)}) \in\mathscr{H}(K),$ there exists $K_{i}\rightarrow K$ in $C^{2}(\mathbb{S}^3),$
$v_{i}\in \mathscr{M}_{K_{i}},$ such that $\{v_{i}\}$ blows up at precisely the $k$ points.
\end{enumerate}
\end{theorem}

\begin{corollary}
For any $k\in \mathbb{N}_{+}$  distinct points $q^{(1)}, \cdots, q^{(k)}\in \mathbb{S}^3,$
  there exists a sequence of Morse functions $\{K_{i}\}\subset \mathscr{A},$ such that
  for some $v_{i}\in \mathscr{M}_{K_{i}},$ $\{v_{i}\}$ blows up at precisely the $k$ points.
\end{corollary}

When  further characterizing the blow up behavior of the solution to \eqref{1.71} (see Section \ref{sec3} below),
we mainly use the Pohozaev type identity  (see Proposition \ref{prop11} below)
and its property (see Proposition \ref{prop2} below) to  judge the sign of the Laplace of the prescribing curvature function at the blow up point. Due to the limitation of the form of  the Pohozaev type identity, the method in this paper is only effective for the case $n-2\sigma=2.$ In a forthcoming paper, we deal with the higher order case, i.e.,
$n=2\sigma+2$ for any $1<\sigma<n/2.$

The paper is organized as follows:
In section \ref{sec2}, we  recall some known results on blow up analysis of the  fractional Nirenberg problem
obtained by Jin-Li-Xiong \cite{JLX}.

 In section \ref{sec3}, our main task is to prove Theorem \ref{thm1} and Theorem \ref{thm2.1}.
By using the method of  subcritical approximation, we obtain Theorem \ref{thm1},
which  further characterizes  the blow up points for solutions  to \eqref{1.71}.
More precisely, we consider the subcritical equation with $\tau>0$ small:
\begin{align}\label{1.92}
P_{\sigma}v=Kv^{2-\tau}, \quad v>0 \quad\hbox{on} \,\, \mathbb{S}^3.
\end{align}
Then we use Theorem \ref{thm1} and some results in \cite{JLX} to prove
Theorem \ref{thm2.1}.

Section \ref{sec4} is devoted to proving the Theorems \ref{thm4}, \ref{thm2}, and \ref{thm5}.
 Firstly, we give the definition of  $\Sigma_{\tau}=\Sigma_{\tau}(\overline{P}_{1},\cdots, \overline{P}_{k}),$
  for $\overline{P}_{1},\cdots, \overline{P}_{k}\in \mathscr{K}^{-}$ with $\mu(M(\overline{P}_{1},\cdots, \overline{P}_{k}))>0.$ Then by using Theorem \ref{thm1}
 and some results in \cite{JLX}, we obtain that for $\tau>0$ very small, the solutions to \eqref{1.92}
 either stay bounded or stay  in one of the $\Sigma_{\tau}$ (see Proposition \ref{prop3} below).
 Furthermore,  we obtain the $H^{\sigma}$ topological degree of the solutions to \eqref{1.92} on
$\Sigma_{\tau}\,($see Theorem \ref{thm3} below$).$
It follows from the above results that  for all $0<\tau<2,$
 the $H^{\sigma}$ total degree of the solutions to \eqref{1.92} is equal
 to $-1$ (see Proposition \ref{prop4} below).
Then we can conclude that $H^{\sigma}$ topological degree of those solutions to \eqref{1.92} which remain bounded as $\tau$ tends to zero is equal to $\mathrm{Index}(K).$
Some well-known results in degree theory imply that the $H^{\sigma}$ degree contribution above is equal
 to the $C^{2,\alpha}$ topological degree of those bounded solutions to \eqref{1.92}.
Thus, we prove Theorem \ref{thm4}.
Furthermore, we complete the proof of Theorem \ref{thm2} by using the degree-counting formula and perturbing the function $K$ near its critical point.
In the end, using Theorem \ref{thm1} and the idea of the proof of Theorem \ref{thm2}, we prove  Theorem \ref{thm5}.

In the Appendix, we provide some useful technical
results and elementary estimates.

Finally, we make some conventions on notation.
Let $\mathbb{R}_+^{n+1}:=\mathbb{R}^{n}\times(0,+\infty).$
For $X=(x,t)\in \mathbb{R}^{n+1}$ and $R\geq 0$, the  symbol
$\mathcal{B}_{R}(X)$
denotes the balls in $\mathbb{R}^{n+1}$ with radius $R$ and center $X,$
and $\mathcal{B}_{R}^{+}(X):= \mathcal{B}_{R}(X)\cap \mathbb{R}_{+}^{n+1}.$
The  symbol  $B_{R}(x)$  denotes the ball in $\mathbb{R}^n$ with
radius $R$ and center $x.$
We also write $\mathcal{B}_{R},$ $\mathcal{B}_{R}^{+},$ $B_{R}$
for $\mathcal{B}_{R}(0),$ $\mathcal{B}_{R}^{+}(0),$ $B_{R}(0),$
respectively.  We always denote by $C$ a  positive constant
which is independent of the main parameters, but it
may vary from line to line.

\section{Quick review of some known facts}\label{sec2}
In this section, we review some results about the blow up analysis of the fractional Nirenberg problem
obtained in Jin-Li-Xiong \cite{JLX}.

Let $\sigma\in (0,1),$  $u_{i} \in C^{2}(\Omega) \cap \dot{H}^{\sigma}(\mathbb{R}^{n})$ with $u_{i} \geq 0$ in $\mathbb{R}^{n}$ satisfy \eqref{1.50} with $K_{i}$ satisfying \eqref{1.49}.
Let $U_i(x,t)$ be the extension of $u_{i}$ as in \eqref{2.8}, we have
\begin{align}\label{1.59}
\begin{cases}
\operatorname{div}\left(t^{1-2 \sigma} \nabla U_{i}\right)=0 & \text { in } \mathbb{R}_{+}^{n+1}, \\ -\lim _{t \rightarrow 0} t^{1-2 \sigma} \partial_{t} U_{i}(x, t)=c_{0} K_{i}(x) U_{i}(x, 0)^{p_{i}}& \text { for any } x \in \Omega,\end{cases}
\end{align}
where $c_{0}=2^{1-2 \sigma}  c(n, \sigma)\Gamma(1-\sigma) / \Gamma(\sigma).$

We say that $U \in H(|t|^{1-2 \sigma}, D)$ if $U \in L^{2}(|t|^{1-2 \sigma}, D),$ and its weak derivatives $\nabla U$ exist and belong to $L^{2}(|t|^{1-2 \sigma}, D).$ The norm of $U$ in $H(|t|^{1-2 \sigma}, D)$ is given by
$$
\|U\|_{H(|t|^{1-2 \sigma}, D)}:=\Big(\int_{D}|t|^{1-2 \sigma} U^{2}\,\ud X+\int_{D}|t|^{1-2 \sigma}|\nabla U|^{2} \,\ud X\Big)^{1/2}.
$$

In the following,
for a domain $D \subset \mathbb{R}^{n+1}$ with boundary $\partial D,$
 we denote by $\partial^{\prime} D$ the interior of $\overline{D} \cap \partial \mathbb{R}_{+}^{n+1}$
 in $\mathbb{R}^{n}=\partial \mathbb{R}_{+}^{n+1},$
  and we set $\partial^{\prime \prime} D=\partial D \backslash \partial^{\prime} D.$

\begin{proposition}[Pohozaev type identity]\label{prop11}
Suppose that $K \in C^{1}(B_{2 R}).$ Let $U \in H(t^{1-2 \sigma}, \mathcal{B}_{2 R}^{+})$ with $U \geq 0$ in $\mathcal{B}_{2 R}^{+}$ be a weak solution of
\begin{align*}
\begin{cases}\operatorname{div}\left(t^{1-2 \sigma} \nabla U\right)=0 & \text { in }\, \mathcal{B}_{2 R}^{+} , \\ -\lim _{t \rightarrow 0} t^{1-2 \sigma} \partial_{t} U(x, t)=K(x) U^{p}(x, 0) & \text { on }\, \partial^{\prime} \mathcal{B}_{2 R}^{+},\end{cases}
\end{align*}
where $p>0 .$ Then
\begin{align*}
\int_{\partial^{\prime} \mathcal{B}_{R}^{+}} B^{\prime}(X, U, \nabla U, R, \sigma)
+\int_{\partial^{\prime \prime} \mathcal{B}_{R}^{+}} t^{1-2 \sigma} B^{\prime \prime}(X, U, \nabla U, R, \sigma)=0,
\end{align*}
where
\begin{align}\label{1.11}
\begin{aligned}
&B^{\prime}(X, U, \nabla U, R, \sigma)=\frac{n-2 \sigma}{2} K U^{p+1}+\langle X, \nabla U\rangle K U^{p}, \\
&B^{\prime \prime}(X, U, \nabla U, R, \sigma)=\frac{n-2 \sigma}{2} U \frac{\partial U}{\partial \nu}-\frac{R}{2}|\nabla U|^{2}+R\Big|\frac{\partial U}{\partial \nu}\Big|^{2}.
\end{aligned}
\end{align}
\end{proposition}

\begin{proposition}\label{prop2}
 Let $\mathcal{M}\in \mathbb{R}$ and $\alpha(X)$ be some differentiable function near the origin with $\alpha(0)=0.$
 Then for $U(X)=|X|^{2\sigma-n}+\mathcal{M}+\alpha(X),$ we have
\begin{align}
\lim _{\delta \rightarrow 0}
\int_{\partial^{\prime \prime}
\mathcal{B}_{\delta}^{+}} t^{1-2 \sigma} B''(X, U, \nabla U, \delta, \sigma)
=-\frac{(n-2\sigma)^2}{2}\mathcal{M} |\mathbb{S}^{n-1}|\frac{\mathrm{B}(n/2,1-\sigma)}{2},
\end{align}
where $\mathrm{B}(\cdot\,,\cdot)$ is the Beta function.
\end{proposition}
\begin{proof}
Since  $U(X)=|X|^{2\sigma-n}+\mathcal{M}+\alpha(X)$, we have
$$
\nabla U(X)=(2\sigma-n) \delta^{2\sigma-n-2}X+\nabla \alpha(X) \quad \text { on }\, \partial'' \mathcal{B}_{\delta}^{+},
$$
and
$$
\frac{\partial U}{\partial\nu}=\nabla U \cdot \nu
=(2\sigma-n)\delta^{2\sigma-n-1} +\frac{\nabla \alpha(X)\cdot X}{\delta}
\quad \text { on }\, \partial'' \mathcal{B}_{\delta}^{+}.
$$
 It follows that
$$
|\nabla U|^{2}=(2\sigma-n)^{2}\delta^{4\sigma-2n-2}
+2(2\sigma-n)\delta^{2\sigma-n-2} \nabla \alpha(X)\cdot X
+|\nabla \alpha(X)|^{2},
$$
and
\begin{align*}
\Big|\frac{\partial U}{\partial\nu}\Big|^{2}
=
(2\sigma-n)^2\delta^{4\sigma-2n-2}
+2(2\sigma-n)\delta^{2\sigma-n-2}\nabla \alpha(X)\cdot X
+|\nabla \alpha(X)\cdot X|^{2}\delta^{-2}
\end{align*}
on $ \partial'' \mathcal{B}_{\delta}^{+}.$
Substituting  the above results into \eqref{1.11}, we can easily obtain
\begin{align*}
&\lim _{\delta \rightarrow 0}
\int_{\partial^{\prime \prime}
\mathcal{B}_{\delta}^{+}} t^{1-2 \sigma} B''(X, U, \nabla U, \delta, \sigma)\\
=&\lim_{\delta\rightarrow 0}
-\frac{(2\sigma-n)^2}{2}\mathcal{M}\delta^{2\sigma-n-1}
\int_{\partial^{\prime \prime}\mathcal{B}_{\delta}^{+}}t^{1-2\sigma}\\
=&-\frac{(2\sigma-n)^2}{2}\mathcal{M}\int_{\partial^{\prime \prime}\mathcal{B}_{1}^{+}}s^{1-2\sigma}\\
=&-\frac{(n-2\sigma)^2}{4}\mathcal{M}|\mathbb{S}^{n-1}|\mathrm{B}(n/2,1-\sigma)|.
\end{align*}
Proposition \ref{prop2} follows from the above.
\end{proof}

\begin{proposition}\label{prop8}
Suppose that for all $\varepsilon \in(0,1),$ $ U \in H(t^{1-2 \sigma},$ $\mathcal{B}_{1}^{+} \backslash \overline{\mathcal{B}_{\varepsilon}^{+}})$
with $U>0$ in $\mathcal{B}_{1}^{+} \backslash \overline{\mathcal{B}_{\varepsilon}^{+}}$ is a weak solution of
$$
\begin{cases}\operatorname{div}\left(t^{1-2 \sigma} \nabla U\right)=0 & \text { in }\, \mathcal{B}_{1}^{+} \backslash \overline{\mathcal{B}_{\varepsilon}^{+}},
\\ -\lim _{t \rightarrow 0} t^{1-2 \sigma} \partial_{t} U(x, t)=0 & \text { on }\, B_{1} \backslash \overline{B_{\varepsilon}^{+}}.
\end{cases}
$$
Then
$$
U(X)=\mathcal{A}|X|^{2 \sigma-n}+\mathcal{W}(X),
$$
where $\mathcal{A}$ is a nonnegative constant and $\mathcal{W} \in H(t^{1-2 \sigma}, \mathcal{B}_{1}^{+})$
satisfies
$$
\begin{cases}
\operatorname{div}\left(t^{1-2 \sigma} \nabla \mathcal{W}\right)=0,
& \text { in }\, \mathcal{B}_{1}^{+}, \\
-\lim _{t \rightarrow 0} t^{1-2 \sigma} \partial_{t}
\mathcal{W}(x, t)=0\, & \text { on }\, B_{1}.
\end{cases}
$$
\end{proposition}

\begin{proposition}\label{prop5}
Suppose that $u_{i} \in C^{2}(\Omega) \cap \dot{H}^{\sigma}(\mathbb{R}^{n})$ with $u_{i} \geq 0$ in $\mathbb{R}^{n}$ satisfies \eqref{1.50} with $K_{i}$ satisfying \eqref{1.49},
 and $y_{i} \rightarrow 0$ is an isolated blow up point of $\{u_{i}\},$ i.e., for some positive constants $A_{3}$ and $\overline{r}$ independent of $i,$
\begin{align}\label{1.57}
|y-y_{i}|^{2 \sigma /(p_{i}-1)} u_{i}(y) \leq A_{3} \quad \text { for all }\, y \in B_{\overline{r}} \subset \Omega.
\end{align}
Denote $U_{i}=\mathcal{P}_{\sigma}[u_{i}]$ and $Y_{i}=(y_{i}, 0).$ Then for any
$0<r<\overline{r}/3,$ we have the following Harnack inequality:
$$
\sup _{\mathcal{B}_{2 r}^{+}(Y_{i}) \backslash
\overline{\mathcal{B}_{r / 2}^{+}(Y_{i})}} U_{i} \leq C  \inf _{\mathcal{B}_{2 r}^{+}
(Y_{i}) \backslash \overline{\mathcal{B}_{r / 2}^{+}(Y_{i})}} U_{i},
$$
where $C$ is a positive constant depending only on $n, \sigma, A_{3}, \overline{r},$ and $\sup _{i}\|K_{i}\|_{L^{\infty}(B_{r}(y_{i}))} .$
\end{proposition}

\begin{proposition}\label{prop6}
Under the hypotheses of Proposition \ref{prop5},
then for any $R_{i} \rightarrow \infty$ and $\varepsilon_{i} \rightarrow 0^{+},$ we have, after passing to a subsequence (still denoted as $\{u_{i}\},\{y_{i}\},$ etc.),
$$
\|m_{i}^{-1} u_{i}(m_{i}^{-(p_{i}-1)/2 \sigma}\cdot+y_{i})-(1+k_{i}|\cdot|^{2})^{(2 \sigma-n)/2} \|_{C^{2}(B_{2 R_{i}}(0))} \leq \varepsilon_{i},
$$
$$
R_{i} m_{i}^{-(p_{i}-1) / 2 \sigma}  \rightarrow 0 \quad \text { as }\, i \rightarrow \infty,
$$
where $m_{i}=u_{i}(y_{i})$ and $k_{i}=K_{i}(y_{i})^{1 / \sigma}/4$.
\end{proposition}

\begin{proposition}\label{prop12}
  Under the hypotheses of Proposition \ref{prop5}, and in addition that $y_{i} \rightarrow 0$ is also an isolated simple blow up point with constant $\rho,$ we have
  \begin{align*}
  \tau_{i}=O(u_{i}(y_{i})^{-2/(n-2\sigma)+o(1)})\quad \hbox{and} \quad
  u_{i}(y_{i})^{\tau_{i}}=1+o(1).
  \end{align*}
Moreover,
$$
u_{i}(y) \leq C u_{i}^{-1}
(y_{i})|y-y_{i}|^{2 \sigma-n} \quad \text { for all } \, |y-y_{i}| \leq 1.
$$
\end{proposition}

\begin{proposition}\label{prop13}
Under the hypotheses of Proposition \ref{prop5}, we have
$$
\int_{|y-y_{i}| \leq r_{i}}|y-y_{i}|^{s}
u_{i}(y)^{p_{i}+1}=
\begin{cases}
O(u_{i}(y_{i})^{-2 s /(n-2 \sigma)}), & -n<s<n, \\ O(u_{i}(y_{i})^{-2 n /(n-2 \sigma)} \log u_{i}(y_{i})), & s=n, \\ o(u_{i}(y_{i})^{-2 n /(n-2 \sigma)}),
& s>n,
\end{cases}
$$
and
$$
\int_{r_{i}<|y-y_{i}| \leq 1}|y-y_{i}|^{s} u_{i}(y)^{p_{i}+1}=
\begin{cases}
o(u_{i}(y_{i})^{-2 s /(n-2 \sigma)}), & -n<s<n, \\ O
(u_{i}(y_{i})^{-2 n /(n-2 \sigma)} \log u_{i}
(y_{i})), & s=n, \\ O(u_{i}(y_{i})^{-2 n /(n-2\sigma )}), & s>n .\end{cases}
$$
\end{proposition}

\section{Compactness of solutions  and characterization of blow up behavior}\label{sec3}
In this section, our main task is to prove  Theorem \ref{thm1} and Theorem \ref{thm2.1}.
We first give the proof of  Theorem \ref{thm1}, which further characterizes
 the blow up points for solutions to \eqref{1.71} and palys a key role in proving Theorem \ref{thm2.1}.
Recall the definitions of the matrix $M$ given in \eqref{M} and its smallest eigenvalue $\mu(M)$.

\begin{proof}[Proof of Theorem \ref{thm1}]
From Jin-Li-Xiong \cite[Theorem 5.3]{JLX}, there exists a constant $\delta^{*}>0$ depending only on
$ \min_{\mathbb{S}^{3}} K$ and $\|K\|_{C^{2}(\mathbb{S}^{3})}$
such that $\{v_{i}\}$ has only isolated simple blow up points $q^{(1)}, \cdots, q^{(k)} \in \mathscr{K}$ $(k \geq 1)$ with $|q^{(j)}-q^{(\ell)}| \geq \delta^{*}$
$(j \neq \ell).$

Under the stereographic projection $F$ with $q^{(j)}$
being the south pole:
\begin{align*}
\begin{aligned}
F:\mathbb{R}^3&\rightarrow \mathbb{S}^3 \backslash \{-q^{(j)}\},\quad
 y\mapsto\Big(\frac{2y}{1+|y|^2},\frac{|y|^2-1}{|y|^2+1}\Big),
\end{aligned}
\end{align*}
the Eq.\,\eqref{1.1} is transformed to
\begin{align}
(-\Delta)^{\sigma} u_{i}(y)=
\widetilde{K}_{i}(y) H_{i}(y)^{\tau_{i}} u_{i}(y)^{p_{i}}, \quad y \in \mathbb{R}^{3},
\end{align}
where
\begin{align}\label{1.3}
u_i(y)=\Big(\frac{2}{1+|y|^2}\Big)v_i(F(y)), \quad \widetilde{K}_{i}(y)=K_{i}(F(y)),
\quad H_i(y)=\frac{2}{1+|y|^2}.
\end{align}
Since $0$ is  an isolated simple  blow up point of $u_{i}.$
Let $U_{i}(Y),$ $Y:=(y,t)\in \mathbb{R}_{+}^{4},$ be the extension of $u_{i}(y)$
and satisfy
\begin{align}\label{eequ}
\begin{cases}\operatorname{div}\left(t^{1-2 \sigma} \nabla U_{i}\right)=0
& \text { in }\, \mathbb{R}_{+}^{4},\\
 -\lim_{t\rightarrow 0} t^{1-2 \sigma} \partial_{t} U_i(y, t)= \widetilde{K}_{i}(y)
  H_{i}(y)^{\tau_{i}} u_{i}(y)^{p_{i}} & \text { for }\, y\in\mathbb{R}^{3}.
 \end{cases}
 \end{align}
Propositions  \ref{prop5}, \ref{prop12}, \ref{prop8},
 and elliptic theory together imply that
\begin{align}
U_{i}(Y_{i}^{(j)}) U_{i}(Y) \rightarrow\mathcal{H}^{(j)}(Y)
:= \mathcal{A}^{(j)}|Y|^{-2}+\mathcal{W}^{(j)}(Y) \quad\text {in} \, C_{\mathrm{loc}}^{2}(\overline{\mathbb{R}_{+}^{4}}
\backslash \{ \cup_{\ell=1}^{k}Y^{(\ell)}\}),\label{1.41}
\end{align}
where $\mathcal{A}^{(j)}>0$ is a constant and
$\mathcal{W}^{(j)}(Y)$ satisfies
\begin{align}\label{1.40}
\begin{cases}\operatorname{div}(t^{1-2\sigma} \nabla \mathcal{W}^{(j)})=0 & \text {in} \, \mathbb{R}_{+}^{4}, \\ -\lim _{t \rightarrow 0} t^{1-2 \sigma} \partial_{t}
\mathcal{W}^{(j)}(y, t)=0 & \text {for}\, y \in \mathbb{R}^{3}.\end{cases}
\end{align}
It follows from the  maximum principle and  the Harnack inequality that
 \begin{align}\label{1.88}
 \wdt{\mathcal{W}}^{(j)}(Y)\equiv 0\quad \hbox{if} \, k=1,\quad
 \wdt{\mathcal{W}}^{(j)}(Y)>0\quad\hbox{if}\, k\geq 2.
 \end{align}

Let's next calculate $\mathcal{A}^{(j)}.$ Multiplying \eqref{eequ} by $U_i(Y_i^{(j)})$
and  integrating by parts
on $\mathcal{B}_1^{+}$
leads to
\begin{align}
0=&\int_{\mathcal{B}_{1}^{+}}U_{i}(Y_{i}^{(j)})\mathrm{div}(\nabla U_{i})\notag\\
=&
\int_{{B}_{1} } u_{i}(y_{i}^{(j)}) \wdt{K}_{i}(y) H_{i}(y)^{\tau_{i}}u_{i}(y)^{p_i}
+\int_{\partial''\mathcal{B}_{1}^{+}}\frac{\partial }{\partial \nu}(U_{i}(Y_{i}^{(j)})U_{i})
=:\mathcal{I}_{1}+\mathcal{I}_{2}.\label{1.21}
 \end{align}
Let  $R_{i}$ be given in Proposition \ref{prop6}, and
\begin{align}\label{rimij}
m_{ij}:=u_i(y_i^{(j)}), \quad
r_i:=R_i(m_{ij})^{-(p_i-1)}.
\end{align}
For $\mathcal{I}_{1},$ from Propositions \ref{prop6}
 and  \ref{prop12} we have
\begin{align}\label{1.62}
\mathcal{I}_{1}
=&
\int_{|y-y_i^{(j)}|\leq r_i}m_{ij}\wdt{K}_{i}(y)
u_i(y)^{p_i}
+\int_{\{|y-y_i^{(j)}|>r_{i}\}\cap\{|y|<1\} }m_{ij}\wdt{K}_{i}(y)H_{i}(y)^{\tau_{i}}
u_i(y)^{p_i}\notag\\
& \quad+O\Big(\tau_{i}\int_{|y-y_i^{(j)}
|\leq r_i}m_{ij}\wdt{K}_{i}(y)
u_i(y)^{p_i}\Big)\notag\\
=&
m_{ij}^{2\tau_{i}}\big(\wdt{K}_{i}(0)+O(|y|)\big)\int_{|x|\leq R_{i}}
\big(m_{ij}^{-1}u_i(m_{ij}^{-(p_{i}-1)}x+y_{i}^{(j)})\big)^{p_{i}}\notag\\
=&2\pi |\mathbb{S}^2|K_{i}(q^{(j)})^{-2}+o(1).
\end{align}
For $\mathcal{I}_{2},$ it follows from\eqref{1.41}  and \eqref{1.40} that
\begin{align}\label{1.20}
\lim_{i\rightarrow \infty}\mathcal{I}_{2}
&=
\int_{\partial ''\mathcal{B}_1^{+}}
\frac{\partial }{\partial \nu}( \mathcal{A}^{(j)}|Y|^{-2}+\mathcal{W}^{(j)}(Y))
=\int_{\partial ''\mathcal{B}_1^{+}}
-2\mathcal{A}^{(j)}
=-\frac{\pi|\mathbb{S}^2|}{2}\mathcal{A}^{(j)}.
\end{align}
By \eqref{1.21}, \eqref{1.62}, and \eqref{1.20}, we conclude that
$\mathcal{A}^{(j)}=4K(q^{(j)})^{-2}.$

From \eqref{1.41}, we have
\begin{align*}
U_{i}(Y_{i}^{(j)}) U_{i}(Y) \rightarrow\mathcal{H}^{(j)}(Y)
:= 4K(q^{(j)})^{-2}|Y|^{-2}+\mathcal{W}^{(j)}(Y) \quad\text {in} \, C_{\mathrm{loc}}^{2}(\overline{\mathbb{R}_{+}^{4}}
\backslash \{ \cup_{\ell=1}^{k}Y^{(\ell)}\}),
\end{align*}
and
\begin{align}\label{1.60}
\begin{aligned}
u_{i}(y_{i}^{(j)}) u_{i}(y) \rightarrow & h^{(j)}(y) \\
:=&  4K(q^{(j)})^{-2}|y|^{-2}+W^{(j)}(y)
\quad \text { in }\, C_{\mathrm{loc}}^{2}(\mathbb{R}^{3}
\backslash\{\cup_{\ell=1}^{k}y^{(\ell)}\}),
\end{aligned}
\end{align}
where  $W^{(j)}(y):=\mathcal{W}^{(j)}(y,0).$

By \eqref{1.3} and $y_{i}^{(j)}\rightarrow 0$ as $i\rightarrow \infty,$ we have
\begin{align*}
\lim_{i\rightarrow \infty}  v_{i}(q_{i}^{(j)}) v_{i}(q)
=\frac{1}{4}\lim_{i\rightarrow \infty}(1+|y|^2)u_{i}(y_{i}^{(j)}) u_{i}(y),
\end{align*}
combining with \eqref{1.60}, it easy to see that for $q \neq q^{(j)}$ and close to $q^{(j)},$
\begin{align}\label{1.4}
\lim _{i \rightarrow \infty} v_{i}(q_{i}^{(j)}) v_{i}(q)
=2G_{q^{(j)}}(q)K(q^{(j)})^{-2}
+\wdt{W}^{(j)}(q)\quad \mbox{in} \,C^{2}_{\mathrm{loc}}
(\mathbb{S}^3\backslash \{\cup_{\ell=1}^{k}q^{(\ell)}\}),
\end{align}
 where $\wdt{W}^{(j)}(q)$ is some regular function  on $\mathbb{S}^{3} \backslash$ $\cup_{\ell \neq j}\{q^{(\ell)}\}$
satisfying $P_{\sigma}\widetilde{W}^{(j)}=0,$
and  $G_{q^{(j)}}(q)$ is the Green function defined as in \eqref{gf}.

When $k\geq 2,$  taking into account the contribution of all
the poles, we deduce
\begin{align}\label{1.87}
\lim _{i \rightarrow \infty} v_{i}(q_{i}^{(j)}) v_{i}(q)
=
2\frac{G_{q^{(j)}}(q)}{K(q^{(j)})^{2}}
+2\sum_{\ell \neq j}
&\lim _{i \rightarrow \infty}
\frac{v_{i}(q_{i}^{(j)})}
{v_{i}(q_{i}^{(\ell)})}
\frac{G_{q^{(\ell)}}(q)}{K(q^{(\ell)})^{2}}\notag\\
&\quad \hbox{in} \,\,C^{2}_{\mathrm{loc}}
(\mathbb{S}^3\backslash \{\cup_{\ell=1}^{k}q^{(\ell)}\}).
\end{align}
In fact, subtracting all the poles from the limit function, we obtain
a regular function  $\wdt{W}_{0}: \mathbb{S}^3\rightarrow \mathbb{R}$
such that $P_{\sigma}\wdt{W}_{0}=0$ on $\mathbb{S}^3,$
so it must be $\wdt{W}_{0}\equiv 0.$

Using \eqref{1.87}, we have,  for $|y|> 0$ small,
\begin{align}\label{1.7}
h^{(j)}(y)=
\frac{4}{K(q^{(j)})^{2}|y|^{2}}+8
\sum_{\ell \neq j} \lim _{i \rightarrow \infty}
 \frac{v_{i}(q_{i}^{(j)})}{v_{i}(q_{i}^{(\ell)})}
\frac{G_{q^{(\ell)}}(q^{(j)})}{K(q^{(\ell)})^{2}}+O(|y|).
\end{align}
The conclusion obtained from the above is easy to see that \eqref{1.74} is true
and \eqref{1.86} can be obtained from Proposition \ref{prop12}.
We have proved Part (ii).

Before stating the result to be proved, we give the following estimates \eqref{1.5.1} and \eqref{1.5}.
Using \cite[Lemmas 4.13 and 4.14]{JLX}, we obtain
\begin{align}\label{1.5.1}
|\nabla K_{i}(y_{i}^{(j)})|
=O(u_{i}(y_{i}^{(j)})^{-1}), \quad \tau_{i}
=O(u_{i}(y_{i}^{(j)})^{-2}),
\end{align}
and from Propositions \ref{prop6}, \ref{prop12}, and \ref{prop13}, we get,
for sufficiently small $\delta>0,$
\begin{align}
\begin{aligned}\label{1.5}
&\sum\limits_{j=1}^{3}\Big|\int_{B_{\delta}} x_{j} u_{i}(y+y_{i}^{(j)})^{p_{i}+1}\Big|
=o(u_{i}(y_{i}^{(j)})^{-1}), \\
&\sum\limits_{j \neq \ell}\Big|\int_{B_{\delta}} x_{j} x_{\ell} u_{i}(y+y_{i}^{(j)})^{p_{i}+1}\Big|
=o(u_{i}(y_{i}^{(j)})^{-2}), \\
&\int_{\partial B_{\delta}} u_{i}(y+y_{i}^{(j)})^{p_{i}+1}
=O(u_{i}(y_{i}^{(j)})^{-p_{i}-1}), \\
&\lim \limits_{i \rightarrow \infty} u_{i}(y_{i}^{(j)})^{2}
\int_{B_{\delta}}|y|^{2} u_{i}(y+y_{i}^{(j)})^{p_{i}+1}
=6\pi|\mathbb{S}^2|K(q^{(j)})^{-5}.
\end{aligned}
\end{align}

Now we give the proof only for the last formula in \eqref{1.5}.
Let $m_{ij}$ and  $r_{i}$ be as in \eqref{rimij}.
Applying Propositions \ref{prop6}, \ref{prop12}, and  \ref{prop13},  we have
\begin{align*}
&m_{ij}^2\int_{|y|\leq \delta}|y|^{2} u_{i}(y+y_{i}^{(j)})^{p_{i}+1}\\
=&m_{ij}^2\int_{|y|\leq r_{i}} |y|^2 u_{i}(y+y_i^{(j)})^{p_{i}+1}
+
m_{ij}^2\int_{r_i<|y|\leq \delta}|y|^2 u_{i}(y+y_i^{(j)})^{p_{i}+1}\\
=&
m_{ij}^{4(2-p_i)}\int_{|x|\leq R_{i}}|x|^2 \Big(m_{ij}^{-1}u_{i}(m_{ij}^{-(p_i-1)}x+y_i^{(j)})\Big)^{p_{i}+1}\\
&\quad
+m_{ij}^2\int_{r_i\leq|x-y_i^{(j)}|\leq \delta}|x-y_{i}^{(j)}|^2u_i(x)^{p_i+1}\\
=&6\pi|\mathbb{S}^2|K(q^{(j)})^{-5}+o(1).
\end{align*}

For any $0<\delta<1,$ combining \eqref{1.5}  with Proposition \ref{prop13}, we can obtain
\begin{align*}
\tau_{i}^{2}\int_{B_{\delta}}
   \wdt{K}_{i}(y+y_{i}^{(j)})
   H_{i}(y+y_{i}^{(j)})^{\tau_{i}}
   u_{i}(y+y_{i}^{(j)})^{p_{i}+1}
   \leq  Cu_{i}(y_{i}^{(j)})^{-4}=o(u_i(y_{i}^{(j)})^{-2}),
\end{align*}
\begin{align*}
\tau_{i}
\int_{B_{\delta}}
   \langle y, \nabla(\wdt{K}_{i}(y+y_{i}^{(j)}) H_{i}(y+y_{i}^{(j)})^{\tau_{i}})
   \rangle
    u_{i}(y+y_{i}^{(j)})^{p_{i}+1}=o(u_i(y_{i}^{(j)})^{-2}),
\end{align*}
and
\begin{align*}
\tau_{i}
 \int_{\partial B_{\delta}}
    \wdt{K}_{i}(y+y_{i}^{(j)})
    H_{i}(y+y_{i}^{(j)})^{\tau_{i}}
    u_{i}(y+y_{i}^{(j)})^{p_{i}+1}=o(u_i(y_{i}^{(j)})^{-2}).
\end{align*}
Then using Proposition \ref{prop13} again, we have
\begin{align}\label{1.22}
&\frac{\tau_i}{3}
   \int_{B_{\delta}}
   \wdt{K}_{i}(y+y_{i}^{(j)})
   (H_{i}(y+y_{i}^{(j)})^{\tau_{i}}-1)
   u_{i}(y+y_{i}^{(j)})^{p_{i}+1}\notag\\
\leq&
C\tau_{i}^2\int_{B_{\delta}}
u_{i}(y)^{p_i+1}=o(u_{i}(y_i^{(j)})^{-2}).
\end{align}

The above estimates, Proposition \ref{prop11}, and \eqref{1.5}  yield,
 for any $0<\delta<1,$
\begin{align}
&\int_{\partial' \mathcal{B}_{\delta}^{+}} B'
(Y, U_{i}(Y+Y_{i}^{(j)}), \nabla U_{i}(Y+Y_{i}^{(j)}), \delta,\sigma)\notag\\
=&
  \int_{B_{\delta}} \wdt{K}_{i}(y+y_{i}^{(j)})
   H_{i}(y+y_{i}^{(j)})^{\tau_{i}}
   u_{i}(y+y_{i}^{(j)})^{p_{i}+1}\notag\\
   &+\int_{B_{\delta}}\langle y, \nabla u_{i}(y+y_{i}^{(j)})
   \rangle \wdt{K}_{i} (y+y_{i}^{(j)})
    H_{i}(y+y_{i}^{(j)})^{\tau_{i}}
   u_{i} (y+y_{i}^{(j)})^{p_{i}}\notag\\
=
 &-\frac{\tau_i}{3}
   \int_{B_{\delta}}
   \wdt{K}_{i}(y+y_{i}^{(j)})
   H_{i}(y+y_{i}^{(j)})^{\tau_{i}}
   u_{i}(y+y_{i}^{(j)})^{p_{i}+1}
 \notag\\
&-\frac{1}{3}
   \int_{B_{\delta}}
   \langle y, \nabla(\wdt{K}_{i}(y+y_{i}^{(j)}) H_{i}(y+y_{i}^{(j)})^{\tau_{i}})
   \rangle
    u_{i}(y+y_{i}^{(j)})^{p_{i}+1}\notag\\
&+\frac{\delta}{3}
    \int_{\partial B_{\delta}}
    \wdt{K}_{i}(y+y_{i}^{(j)})
    H_{i}(y+y_{i}^{(j)})^{\tau_{i}}
    u_{i}(y+y_{i}^{(j)})^{p_{i}+1}\notag\\
   & +o(u_{i}(y_i^{(j)})^{-2}) \notag\\
=:&\mathcal{J}_1+\mathcal{J}_{2}+\mathcal{J}_{3}+o(u_{i}(y_i^{(j)})^{-2}),\label{1.42}
\end{align}
where in the  first equality, we take advantage of the fact that
the Taylor expansion:
\begin{align*}
\frac{1}{p_i+1}=\frac{1}{3-\tau_{i}}=\frac{1}{3(1-\tau_i/3)}
=\frac{1}{3}\Big(1+\frac{\tau_i}{3}+O(\tau_i^2)\Big).
\end{align*}

It follows  from  \eqref{1.5}, \eqref{1.22},
Proposition \ref{prop6}, and Proposition \ref{prop13} that
\begin{align}
\mathcal{J}_1=&
-\frac{\tau_i}{3}
   \int_{B_{\delta}}
   \wdt{K}_{i}(y+y_{i}^{(j)})
   u_{i}(y+y_{i}^{(j)})^{p_{i}+1}+o(m_{ij}^{-2})\notag\\
=&-\frac{\tau_{i}}{3}
\frac{2^3}{K(q^{(j)})^{2}}
\int_{\mathbb{R}^3}  \frac{1}{(1+|z|^2)^3}+o(m_{ij}^{-2})\notag\\
=&-\frac{\pi|\mathbb{S}^2|}{6}\frac{\tau_{i}}{K(q^{(j)})^2}
+o(m_{ij}^{-2}).\label{1.44}
\end{align}
Applying Proposition \ref{prop13} and \eqref{1.5}, we conclude that
\begin{align}
\mathcal{J}_2=&
-\frac{1}{3}
   \int_{B_{\delta}}
   \langle
   y, \nabla\big(\wdt{K}_{i}(y+y_{i}^{(j)}) H_{i}(y+y_{i}^{(j)})^{\tau_{i}}\big)
   \rangle
    u_{i}(y+y_{i}^{(j)})^{p_{i}+1}+o(m_{ij}^{-2})\notag\\
 =&
 -\frac{1}{3}
 \int_{B_{\delta}}
   \langle y, \nabla(\wdt{K}_{i}(y+y_{i}^{(j)})) H_{i}(y+y_{i}^{(j)})^{\tau_{i}}
 \rangle
    u_{i}(y+y_{i}^{(j)})^{p_{i}+1}\notag\\
   & \quad-\frac{1}{3}\int_{B_{\delta}}
   \langle y, \wdt{K}_{i}(y+y_{i}^{(j)}) \nabla(H_{i}(y+y_{i}^{(j)})^{\tau_{i}})
   \rangle
    u_{i}(y+y_{i}^{(j)})^{p_{i}+1}+o(m_{ij}^{-2})\notag\\
 =&
 -\frac{1}{3}
 \sum\limits_{\ell}\int_{B_{\delta}}
 y_{\ell} \frac{\partial \wdt{K}_{i}}{\partial y_{\ell}}(y+y_{i}^{(j)})
    u_{i}(y+y_{i}^{(j)})^{p_{i}+1}+o(m_{ij}^{-2})\notag\\
 =&
 -\frac{1}{3}\int_{B_{\delta}}
  y\cdot \nabla \wdt{K}_{i}(y_{i}^{(j)})u_{i}(y+y_{i}^{(j)})^{p_{i}+1}\notag\\
  &\quad-\frac{1}{3}\sum\limits_{\ell,m}
  \int_{B_{\delta}}
  y_{\ell}y_{m}\frac{\partial^2\wdt{K}_{i}}{\partial y_{\ell}\partial y_{m}}(y_{i}^{(j)})
  u_{i}(y+y_{i}^{(j)})^{p_{i}+1}+o(m_{ij}^{-2})\notag\\
 =&
 -\frac{1}{9}\Delta \wdt{K}(0)
   \int_{B_{\delta}}|y|^2u_{i}(y+y_{i}^{(j)})^{p_{i}+1}+o(m_{ij}^{-2})\notag\\
  = &
  -\frac{4}{9}\Delta_{g_0} K(q^{(j)})
   \int_{B_{\delta}}|y|^2u_{i}(y+y_{i}^{(j)})^{p_{i}+1}+o(m_{ij}^{-2}),\label{1.43}
\end{align}
where we  used the definition of the Laplace-Beltrami operator in the last equality.
By \eqref{1.5}, we get
\begin{align}
\mathcal{J}_{3}=\frac{\delta}{3}
    \int_{\partial B_{\delta}}
    \wdt{K}_{i}(y+y_{i}^{(j)})
    H_{i}(y+y_{i}^{(j)})^{\tau_{i}}
    u_{i}(y+y_{i}^{(j)})^{p_{i}+1}
    =o(m_{ij}^{-2}).\label{1.45}
\end{align}

It follows  from Proposition \ref{prop11} and \eqref{1.42}--\eqref{1.45} that
\begin{align}
&\int_{\partial'' \mathcal{B}_{\delta}^{+}} B''(Y, U_{i}(Y+Y_{i}^{(j)}),
\nabla U_{i}(Y+Y_{i}^{(j)}), \delta, \sigma)\notag\\
=&-\int_{\partial' \mathcal{B}_{\delta}^{+}} B'
(Y, U_{i}(Y+Y_{i}^{(j)}), \nabla U_{i}(Y+Y_{i}^{(j)}), \delta,\sigma)\notag\\
=&\frac{\pi|\mathbb{S}|^2}{6}\frac{\tau_{i}}{K(q^{(j)})^2}
 +\frac{4}{9}\Delta_{g_{0}} K(q^{(j)})
 \int_{B_{\delta}}|y|^2u_{i}(y+y_{i}^{(j)})^{p_{i}+1}
   +o(m_{ij}^{-2}).\label{1.61}
\end{align}
By \eqref{1.3} and the definition of $\mu^{(j)}$, we have
\begin{align*}
\mu^{(j)}=\lim_{i\rightarrow \infty}\tau_i
v_{i}(q_{i}^{(j)})^2
=\lim_{i\rightarrow\infty}\frac{1}{4}\tau_{i}u_{i}(y_{i}^{(j)})^2.
\end{align*}
Thus,
multiplying \eqref{1.61} by $U_{i}(Y_{i}^{(j)})^{2}$ and sending $i$ to $\infty$,
and  using Proposition \ref{prop11} and \eqref{1.5}, we conclude  that
\begin{align*}
&\int_{\partial'' \mathcal{B}_{\delta}^{+}} B''
(Y, \mathcal{H}^{(j)}(Y+Y_{i}^{(j)}), \nabla \mathcal{H}^{(j)}(Y+Y_{i}^{(j)}), \delta,\sigma)\\
=&\frac{8\pi|\mathbb{S}^2|\Delta_{g_0} K(q^{(j)})}{3K(q^{(j)})^5}
+\frac{2\pi|\mathbb{S}^2|\mu^{(j)}}{3K(q^{(j)})^2}.
\end{align*}
Let $\delta\rightarrow 0$, it follows from Proposition \ref{prop2} that
\begin{align}\label{1.9}
\begin{aligned}
\frac{8\pi|\mathbb{S}^2|\Delta_{g_0} K(q^{(j)})}{3K(q^{(j)})^5}
+\frac{2\pi|\mathbb{S}^2|\mu^{(j)}}{3K(q^{(j)})^2}
=-\frac{2\pi|\mathbb{S}^2|W^{(j)}(0)}{K(q^{(j)})^{2}}.
\end{aligned}
\end{align}
Consequently,  we have $q^{(j)} \in \mathscr{K} \backslash \mathscr{K}^{+},$ $\forall \,1 \leq j \leq k,$
 and when $k \geq 2,$ $q^{(j)} \in \mathscr{K}^{-},$  $ \forall \,1 \leq j \leq k.$

It is easy to see that \eqref{1.6} follows from \eqref{1.88} and \eqref{1.9} when $k=1$.
When $k \geq 2,$
by \eqref{1.7} we have
\begin{align}\label{1.8}
W^{(j)}(0)=8
\sum_{\ell \neq j} \frac{\lambda_{\ell}}{\lambda_{j}}
\frac{G_{q^{(\ell)}}(q^{(j)})}{K(q^{(j)})K(q^{(\ell)})},\quad \forall\, 1\leq j\leq k.
\end{align}
Substituting \eqref{1.8} into \eqref{1.9}, we  get
\begin{align*}
-6\sum_{\ell \neq j}
\frac{G_{q^{(\ell)}}(q^{(j)})}{K(q^{(j)})K(q^{(\ell)})}\lambda_{\ell}
-\frac{\Delta_{g_0} K(q^{(j)})}{K(q^{(j)})^3}\lambda_{j}
=\frac{1}{4}\lambda_{j}\mu^{(j)}.
\end{align*}
We have established \eqref{1.10} and thus verified Part (iii).

We claim that there exists some
\begin{align}\label{eig}
\eta=(\eta_{1},\cdots,\eta_{k})\neq 0\quad
\hbox{with}\quad\eta_{\ell}\geq 0,  \,\forall\, \ell=1,\cdots,k,
 \end{align}
such that
\begin{align*}
\sum_{\ell=1}^{k} M_{\ell j}(q^{(1)},
 \cdots, q^{(k)}) \eta_{\ell}=\mu(M) \eta_{j}, \quad \forall\, j=1,\cdots,k.
\end{align*}
Indeed, choose $\Lambda>\max_{i}M_{ii},$ then the matrix $\Lambda I-M$ is a positive
matrix (see \cite{HJ} for the definition), where $I$ denotes the unit matrix.
The claim follows from \cite[Theorem 8.2.2]{HJ}.

Multiplying \eqref{1.10} by $\eta_{j}$
and summing over $j,$  and using Part (ii) and  \eqref{eig}, we have
\begin{align}\label{1.84}
\mu(M) \sum_{j} \lambda_{j} \eta_{j}
=\sum_{\ell, j} M_{\ell j} \lambda_{\ell} \eta_{j}
=\frac{1}{4} \sum_{j} \lambda_{j} \eta_{j} \mu^{(j)}\geq 0.
\end{align}
It follows that $\mu(M)\geq 0$. We have verified Part (i).
 Part (iv) follows from Parts (i)--(iii). The proof of Theorem \ref{thm1} is completed.
\end{proof}

Using Theorem \ref{thm1}, we can give the proof of Theorem \ref{thm2.1}.
\begin{proof}[Proof of Theorem \ref{thm2.1}]
We first  prove the upper bounds in \eqref{1.93}.
Suppose this assertion of the theorem  is false. Then we can find
 that there exists $K_{i} \rightarrow K$ in $C^{2}(\mathbb{S}^{3})$ such that
$\max _{\mathbb{S}^{3}} v_{i} \rightarrow \infty$ for some $v_{i} \in \mathscr{M}_{K_{i}}.$
 Theorem \ref{thm1} shows that $\{v_{i}\}$ has only isolated simple
blow up points $\{q^{(1)}, \cdots, q^{(k)}\}.$
It follows from \cite[Theorem 5.5]{JLX}  that $k>1.$
Using Part (i) of Theorem \ref{thm1}, we obtain $q^{(i)},\cdots, q^{(k)}\in \mathscr{K}^{-}.$

Applying Theorem \ref{thm1} with $\tau_{i}=0,$
we deduce  that $q^{(1)}, \cdots, q^{(k)} \in \mathscr{K}^{-}$ and
 for all $1 \leq j \leq k,$ $\sum_{\ell=1}^{k} M_{\ell j} \lambda_{\ell}=0,$
 where $\lambda_{\ell}>0,$ $\ell=1, \cdots, k.$

Analysis similar to that in the proof of Theorem \ref{thm1} shows
that $\mu(M)$ has at least one nonnegative eigenvector $\eta=(\eta_{1},\cdots,\eta_{k})$ as in \eqref{eig},
then we have
\begin{align*}
\mu(M) \sum_{j} \lambda_{j} \eta_{j}
=\sum_{\ell, j} M_{\ell j} \lambda_{\ell} \eta_{j}
=0.
\end{align*}
It follows that $\mu(M)=0.$  This leads to a contradiction with $K \in \mathscr{A}.$
Then by the Harnack inequality in \cite[Lemma 4.3]{JLX} and
Schauder estimates in \cite[Theorem 2.11]{JLX}, we
complete the proof of  Theorem \ref{thm2.1}.
\end{proof}

\section{The existence results on $\mathbb{S}^3$}\label{sec4}

In this section, we first prove Theorem \ref{thm4},  which is about the degree-counting formula and the existence
of the solutions. Before that, we prove that as $\tau \rightarrow 0^{+},$ the solutions to
 the subcritical equation (see \eqref{subequ} below)  either stay bounded and
converge  to the solutions to critical equations \eqref{1.71} in $C^{2}$ norm or become unbounded and blow up at finite points.

Then by using Theorem \ref{thm4} and perturbing the prescribing function near its critical point,
we can know  exactly where the blow up occur when $K\notin \mathscr{A}.$
From Theorem \ref{thm1} and the proof of Theorem \ref{thm2}, we show Theorem \ref{thm5} holds.
\subsection{On the case of subcritical equations}
In this subsection,  we consider the following  subcritical equation:
\begin{align}\label{subequ}
P_{\sigma}v=Kv^{2-\tau}\quad \hbox{on} \,\, \mathbb{S}^3,
\end{align}
where $K\in C^{2}(\mathbb{S}^3)$ and $\tau>0.$

Denote the  $H^{\sigma}(\mathbb{S}^3)$ inner product and norm by
\begin{align*}
\langle u, v\rangle=\int_{\mathbb{S}^{n}}(P_{\sigma} u) v,
\quad\|u\|_{\sigma}=\sqrt{\langle u, u\rangle}.
\end{align*}
The  Euler-Lagrange functional associated  with \eqref{subequ} is
\begin{align}\label{fun}
I_{\tau}(v)=\frac{1}{2}\int_{\mathbb{S}^3}(P_{\sigma} v) v-\frac{1}{3-\tau}\int_{\mathbb{S}^3}K|v|^{3-\tau},\quad \forall\, v\in H^{\sigma}(\mathbb{S}^3).
\end{align}

\begin{definition}\label{1.89}
Let $K\in C^{2}(\mathbb{S}^3),$ $\mathscr{K}^{-}$ be as in \eqref{1.65} and $k\in \mathbb{N}_{+}.$
Let $\overline{P}_{1}, \cdots, \overline{P}_{k} \in \mathscr{K}^{-}$ be the critical points
of $K$ with $\mu(M(\overline{P}_{1}, \cdots, \overline{P}_{k}))>0$
and  $\varepsilon_{0}>0$ be sufficiently small. Define
\begin{align*}
\begin{aligned}
\Omega_{\varepsilon_{0}}=&\Omega_{\varepsilon_{0}}
(\overline{P}_{1}, \cdots, \overline{P}_{k}) \\
=&\left\{(\alpha, t, P) \in \mathbb{R}_{+}^{k} \times
\mathbb{R}_{+}^{k} \times(\mathbb{S}^{3})^{k}:|\alpha_{i}-1/K(P_{i})|<\varepsilon_{0},\right.\\
&\quad\left. t_{i}>1/\varepsilon_{0},\, |P_{i}-\overline{P}_{i}|<\varepsilon_{0},
1 \leq i \leq k\right\}.
\end{aligned}
\end{align*}
\end{definition}

It is well known  that for $P\in \mathbb{S}^3$ and $t>0,$
\begin{align}\label{delta}
\delta_{P, t}(x):=\frac{t}
{1+\frac{t^{2}-1}{2}(1-\cos\, d(x, P))},\quad \forall\, x\in \mathbb{S}^3
\end{align}
 is a family of positive solutions to
\begin{align}\label{1.51}
P_{\sigma} v=v^2, \quad v>0 \quad \text { on }\, \mathbb{S}^{3},
\end{align}
where $d(\cdot\,, \,\cdot )$ is the distance induced by the standard metric of $\mathbb{S}^{3}.$

Using the idea introduced in  \cite{Chti1} and
 \cite{Ba}, we have the following lemma.
 \begin{lemma}
 Let $\varepsilon_{0}$ be sufficiently small and $\Omega_{\varepsilon_{0}}=\Omega_{\varepsilon_{0}}(\overline{P}_{1}, \cdots, \overline{P}_{k})$ be as in Definition \ref{1.89}. For any $u\in H^{\sigma}(\mathbb{S}^3)$
satisfying the inequality
\begin{align*}
\Big\|u-\sum_{i=1}^k\wdt{\alpha}_{i}\delta_{\wdt{P}_{i},\wdt{t}_{i}}
\Big\|_{\sigma}<\frac{\varepsilon_{0}}{2}
\end{align*}
 for some $(\wdt{\alpha},\wdt{t},\wdt{P})\in \Omega_{\varepsilon_0/2},$
then there exists a unique $(\alpha,t,P)\in \Omega_{\varepsilon_{0}}$ such that
\begin{align*}
u=\sum_{i=1}^{k}\alp_{i}\delta_{P_{i},t_{i}}+v,
\end{align*}
with $v$
satisfies
 \begin{align}\label{1.12}
\langle v, \delta_{P_{i}, t_{i}}\rangle=\big\langle v, \frac{\partial \delta_{P_{i}, t_{i}}}{\partial P_{i}^{(\ell)}}\big\rangle=\big\langle v, \frac{\partial \delta_{P_{i}, t_{i}}}{\partial t_{i}}\big\rangle=0,
\end{align}
where $\frac{\partial}{\partial P_{i}^{(\ell)}}$ denotes the corresponding derivatives.
\end{lemma}
We denote the set of $v \in H^{\sigma}(\mathbb{S}^{3})$ satisfying \eqref{1.12} by $E_{P,t}.$
In what follows, we work in some orthonormal basis near $\{\overline{P}_{1},\cdots,\overline{P}_{k}\}.$
\begin{definition}
Let $A$ be sufficiently large, $\varepsilon_{0}, \nu_{0}>0$ be sufficiently small,
$k\in \mathbb{N}_{+},$
and $\Omega_{\varepsilon_{0}/2}=\Omega_{\varepsilon_{0}/2}(\overline{P}_{1},\cdots,\overline{P}_{k})$
be as in Definition \ref{1.89}. Define
\begin{align}\label{stau}
\begin{aligned}
& \Sigma_{\tau}(\overline{P}_{1}, \cdots, \overline{P}_{k}) \\
=&\left\{(\alpha, t, P, v)
\in \Omega_{\varepsilon_{0}/2}\times H^{\sigma}(\mathbb{S}^{3}):\right.\\
&\quad\left.|P_{i}-\overline{P}_{i}|<\tau^{1/2}|\log \tau|,\,
A^{-1} \tau^{-1 / 2}<t_{i}<A \tau^{-1 / 2},\, v \in E_{P,t},\,\|v\|<\nu_{0}\right\}.
\end{aligned}
\end{align}
Without confusion we use the same notation for
$$
\Sigma_{\tau}(\overline{P}_{1},\cdots,\overline{P}_{k})=\Big\{u=\sum_{i=1}^{k} \alpha_{i} \delta_{P_{i}, t_{i}}+v:(\alpha, t, P, v)
 \in \Sigma_{\tau}\Big\} \subset H^{\sigma}(\mathbb{S}^{3}).
$$
\end{definition}

\begin{remark}
Due to  Theorem \ref{thm2.1}, we only need to prove Theorem \ref{thm4}
for $K\in \mathscr{A}$ being a Morse function. Once this is achieved,  we also prove that
the $\mathrm{Index}$ as in Definition \ref{defn1.2} is well defined on $\mathscr{A}.$
\end{remark}

Using blow up analysis, we first give the necessary conditions on blowing up solutions to \eqref{subequ} when $\tau$ tends to $0.$
\begin{proposition}\label{prop3}
Let $K \in \mathscr{A}$ be a  Morse function and $\mathscr{K}^{-}$ be as in \eqref{1.65}.
 Then for any $\alpha \in (0,1),$ there exists some
positive constants $\varepsilon_{0},\nu_{0} \ll 1,$ and  $A, R \gg 1$ depending only on $K,$ such that
 when $\tau>0$ is sufficiently small,  for all u satisfying $u \in H^{\sigma}
(\mathbb{S}^{3}),$ $u>0,$ $I_{\tau}^{\prime}(u)=0,$ we have
 \begin{align*}
u \in \mathscr{O}_{R} \cup\{\cup_{k \geq 1}
\cup_{\overline{P}_{1}, \cdots, \overline{P}_{k} \in \mathscr{K}^{-},
\mu(M(\overline{P}_{1}, \cdots, \overline{P}_{k}))>0}
 \Sigma_{\tau}(\overline{P}_{1}, \cdots, \overline{P}_{k})\},
 \end{align*}
where $I_{\tau}'(u)$ is as in \eqref{subequ}, $\mathscr{O}_{R}$ is as in \eqref{1.82}
and $\Sigma_{\tau}(\overline{P}_{1},\cdots,\overline{P}_{k})$ is as in \eqref{stau}.
\end{proposition}
\begin{proof}
For any $\tau>0$  sufficiently small, let $u_{\tau}\in H^{\sigma}(\mathbb{S}^3),$ $u_{\tau}>0$
be a critiacl point of $I_{\tau}(u).$ If $u_{\tau}$ is uniformly bounded,
then there exists a $R>0$ such that $u_{\tau}\in \mathscr{O}_{R},$
and the proof is now completed.
If not,   there exists $\tau_{i}\rightarrow 0$ such that $u_{\tau_{i}}\rightarrow \infty.$
It follows from Theorem \ref{thm1} and $K\in \mathscr{A}$ that
there exists a constant $\delta^{*}>0$ such that
$\{u_{\tau_{i}}\}$ has only isolated simple blow up points
$q^{(1)},\cdots, q^{(k)}\in \mathscr{K}^{-},$ with $|q^{(j)}-q^{(\ell)}|\geq \delta^{*}
,$ $\forall j\ne \ell,$ and $\mu (q^{(1)},\cdots, q^{(k)})>0.$
Then Proposition \ref{prop3} can be deduced from Propositions \ref{prop8}, \ref{prop5},
\ref{prop12},
and elliptic theory.
\end{proof}

Now we are going to show that if $K\in \mathscr{A}$ is a Morse function,
 one can construct solutions highly concentrating at arbitrary points $q^{(1)}, \cdots, q^{(k)} \in \mathscr{K}^{-}$ provided $\mu(M(q^{(1)}, \cdots, q^{(k)}))>0.$

\begin{theorem}\label{thm3}
Let $ K \in \mathscr{A}$ be a Morse function and $\mathscr{K}^{-}$ be as in \eqref{1.65}.
Let $\tau, \varepsilon_{0},\nu_{0}>0$ be sufficiently  small,
$A>0$ be sufficiently  large and $k\in \mathbb{N}_{+}.$
Then for any $\overline{P}_1,\cdots,\overline{P}_{k}\in \mathscr{K}^{-}$
satisfying  $\mu(M(\overline{P}_{1},\cdots,\overline{P}_{k}))>0,$ we have
\begin{align}\label{1.75}
\deg_{H^{\sigma}}
\left(u-P_{\sigma}^{-1}(K|u|^{1-\tau}u),
\Sigma_{\tau}(\overline{P}_{1},\cdots,\overline{P}_{k}), 0\right)
=(-1)^{k+\sum_{j=1}^{k}i(\overline{P}_{j})},
\end{align}
where $\deg_{H^{\sigma}}$ denotes the Leray-Schauder degree in $H^{\sigma}(\mathbb{S}^3),$
and $i(\overline{P}_{j})$ is the Morse index of $K$ at $\overline{P}_{j}.$
\end{theorem}

The following conclusion is needed for proving Theorem \ref{thm3}.
\begin{proposition}\label{prop1}
Under the assumptions of the Theorem \ref{thm3},
in addition that $\Sigma_{\tau}(\overline{P}_{1},\cdots,\overline{P}_{k})$ is as in \eqref{1.75}
and $E_{P,t}$ is as in \eqref{1.12} for the given $(\alpha, t,P).$
 Then there exists a unique minimizer $\overline{v}=\overline{v}_{\tau}(\alpha, t, P)
  \in E_{P,t}$ of  $I_{\tau}(\sum_{i=1}^{k} \alpha_{i} \delta_{P_{i}, t_{i}}+v)$
   with respect to $\{v \in E_{P,t}:\|v\|_{\sigma}<\nu_{0}\}.$ Furthermore,
   there exists a constant $C$ independent of $\tau$ such that
\begin{align*}
&\|\overline{v}\|_{\sigma}\leq
C\sum_{i=1}^{k}|\nabla K(P_i)|\tau^{1/2}+C\tau|\log \tau|\leq C\tau|\log \tau|.
\end{align*}
\end{proposition}

 \begin{proof}
 For $(\alpha, t, P, v) \in \Sigma_{\tau}(\overline{P}_{1},\cdots,\overline{P}_{k}),$
 which is  simply written as $\Sigma_{\tau},$ it follows from \eqref{1.12} that
 \begin{align}
 \begin{aligned}\label{1.26}
 &I_{\tau}\Big(\sum_{i=1}^{k}\alpha_{i}\delta_{P_i,t_i}+v\Big)\\
=&\frac{1}{2}\sum_{i=1}^{k}\alpha_{i}^2\int_{\mathbb{S}^3}\delta_{P_i,t_i}^{3}
+\frac{1}{2}\sum_{j\ne i}\alpha_{i}\alpha_{j}\int_{\mathbb{S}^3}\delta_{P_i,t_i}^{2}\delta_{P_j,t_j}
+\frac{1}{2}\int_{\mathbb{S}^3}(P_{\sigma} v) v
\\
&\quad
+\frac{1}{3-\tau}
\int_{\mathbb{S}^3}K\Big|\sum_{i=1}^{k}\alpha_i\delta_{P_i,t_i}+v\Big|^{3-\tau}.
\end{aligned}
 \end{align}
Using Lemma \ref{cplema1} and \eqref{p1p1}, we have,
\begin{align*}
&I_{\tau}\Big(\sum_{i=1}^{k}\alpha_{i}\delta_{P_i,t_i}+v\Big)\\
=&\frac{|\mathbb{S}^3|}{2}\sum_{i=1}^{k} \alpha_{i}^{2}
+\sum_{i<j} \alpha_{i} \alpha_{j}
\int_{\mathbb{S}^{3}}\delta_{P_{i}, t_{i}}^{2} \delta_{P_{j}, t_{j}}\\
\quad&-\frac{1}{3-\tau} \int_{\mathbb{S}^{3}} K\Big(\sum_{i=1}^{k} \alpha_{i} \delta_{P_{i}, t_{i}}\Big)^{3-\tau}
-\int_{\mathbb{S}^3}K\Big(\sum_{i=1}^{k}\alpha_i \delta_{P_i,t_i}\Big)^{2-\tau}v
\\
\quad&
+\frac{1}{2}\int_{\mathbb{S}^3}(P_{\sigma} v) v
-\frac{2-\tau}{2}\int_{\mathbb{S}^3}\Big( \sum_{i=1}^{k}\alpha_{i}\delta_{P_i,t_i}\Big)^{1-\tau}v^2
+V(\tau, \alpha, t, P, v),
\end{align*}
where $|V(\tau, \alpha, t, P, v)|
\leq C\|v\|_{\sigma}^{3-\tau}$ and $C$ depends only on $K, \nu_{0},$ and $A$.

For $\varphi, v \in E_{P,t}$, set
\begin{align}
f_{\tau}(v) &=-\int_{\mathbb{S}^{3}} K\big(\sum_{i=1}^{k} \alpha_{i} \delta_{P_{i}, t_{i}}\big)^{2-\tau} v,\label{1.24}\\
Q_{\tau}(v,\varphi)&=\frac{1}{2}\int_{\mathbb{S}^3}(P_{\sigma}v) \varphi
-\frac{2-\tau}{2}\int_{\mathbb{S}^3}K\big(\sum_{i=1}^{k}\alp_{i}\delta_{P_i,t_i}\big)^{1-\tau}\,
v\varphi,\label{1.94}\\
Q_{0}(v,\varphi)&=\frac{1}{2}\int_{\mathbb{S}^3}
(P_{\sigma} v)\varphi-\int_{\mathbb{S}^3}\big(\sum_{i=1}^{k}\alpha_{i}\delta_{P_i,t_i}\big) v\varphi.
\end{align}
A direct calculation gives, for all $\varphi \in E_{P,t}$,
\begin{align}\label{1.66}
I_{\tau}'\Big(\sum_{i=1}^{k}\alpha_{i}\delta_{P_{i},t_{i}}+v\Big)\varphi
=f_{\tau}(\varphi)+2Q_{\tau}(v,\varphi)+
\langle V_{v}(\tau,\alpha,t,P,v),\varphi\rangle,
\end{align}
where $V_{v}$ is some function satisfying
$\|V_{v}(\tau,\alpha,t,P,v)\|_{\sigma}\leq C\|v\|_{\sigma}^{2-\tau}.$

 Since $f_{\tau}$ is a continuous linear functional over $E_{P,t}$,
 there exists   a unique $\widetilde{f}_{\tau}\in E_{P,t}$ such that
 \begin{align}\label{1.47}
 f_{\tau}(\varphi)=\langle \widetilde{f}_{\tau}, \varphi\rangle,\quad\forall\,\varphi\in E_{P,t}.
\end{align}
It is proved in \cite{Liu_cpaa} that there exists a constant $\delta_{0}>0$ (independent of $\tau$)
such that
\begin{align*}
Q_{0}(v,v)\geq\delta_0\|v\|_{\sigma}^2, \quad \forall\, (\alpha,t,P,v)\in \Sigma_{\tau}.
\end{align*}
We choose $\varepsilon_{0}$ sufficiently small from the beginning. Using some elementary
estimates as in Appendix, we have, for $\tau>0$ small,
\begin{align}\label{1.68}
Q_{\tau}(v,v)\geq \frac{\delta_{0}}{2}\|v\|_{\sigma}^{2},\quad \forall \, (\alpha,t,P,v)\in \Sigma_{\tau}.
\end{align}
Thus,  there  exists a unique symmetric continuous and coercive operator $\widetilde{Q}_{\tau}$ from
$E_{P,t}$ onto itself such that, for any $\varphi\in E_{P,t},$
\begin{align}\label{1.67}
Q_{\tau}(v, \varphi)=\langle \widetilde{Q}_{\tau}v, \varphi\rangle.
\end{align}
Using these notations, \eqref{1.66}, \eqref{1.47}, and \eqref{1.67}, we have
\begin{align}\label{1.96}
I_{\tau}'\Big(\sum_{i=1}^{k}\alpha_{i}\delta_{P_{i},t_{i}}+v\Big)
=\widetilde{f}_{\tau}+2\widetilde{Q}_{\tau}v+V_{v}(\tau,\alpha,t,P,v).
\end{align}

There is an equivalence between the existence of minimizer $\overline{v}_{\tau}$ and
\begin{align}\label{1.46}
\widetilde{f}_{\tau}+2\widetilde{Q}_{\tau}v+V_{v}(\tau,\alpha,t,P,v)=0,\quad v\in E_{P,t}.
\end{align}
As in \cite{Liu_cpaa,rey}, by the implicit function theorem, there exists a unique $v_{\tau}\in E_{P,t}$ with
$\|v\|_{\sigma}<\nu_{0}$ satisfying  \eqref{1.46} and
\begin{align}\label{1.48}
\|\overline{v}\|_{\sigma}\leq C\|\widetilde{f}_{\tau}\|_{\sigma}.
\end{align}
Thus, we only need to estimate $\|\widetilde{f}_{\tau}\|_{\sigma}.$
From  Lemma \ref{lem2}, \eqref{1-tau},  and \eqref{p-p1}, we can obtian
\begin{align*}
f_{\tau}(v)=
&-\int_{\mathbb{S}^{3}} K\Big(\sum_{i=1}^{k} \alpha_{i}^{2-\tau} \delta_{P_{i}, t_{i}}^{2-\tau}\Big) v
+O\Big(\sum_{i \neq j} \int_{\mathbb{S}^{3}} \delta_{P_{i}, t_{i}}^{1-\tau} \delta_{P_{j}, t_{j}}|v|\Big) \notag\\
=&
-\int_{\mathbb{S}^{3}}(K-K(P_{i})) \sum_{i=1}^{k} \alpha_{i}^{2-\tau}
\delta_{P_{i}, t_{i}}^{2} v+O\Big(\sum_{i=1}^{k} \int_{\mathbb{S}^{3}}
|\delta_{P_{i}, t_{i}}^{2-\tau}-\delta_{P_{i}, t_{i}}^{2} | |v|\Big)\notag \\
&\quad+O\Big(
\sum_{i\ne j}\|\delta_{P_i,t_i}^{1-\tau}\delta_{P_j,t_j}\|_{L^{3/2}(\mathbb{S}^3)}\|v\|_{\sigma}
\Big) \notag\\
=& O\Big(\sum_{i=1}^{k}|\nabla_{g_0} K(P_{i})| \int_{\mathbb{S}^{3}}|P-P_{i}| \delta_{P_{i}, t_{i}}^{2}|v|\Big)
+O\Big(\sum_{i=1}^k\int_{\mathbb{S}^3}|P-P_{i}|^2\delta_{P_i,t_{i}}^2|v|\Big)\notag\\
&\quad+O(\tau|\log\tau|\|v\|_{\sigma}),
\end{align*}
where $|P-P_{i}|$ represents the distance between  two points $P$
 and $P_{i}$  after through a stereographic  projection.

Using \eqref{p-p1} again,
we have, for all $(\alpha,t,P,v)\in  \Sigma_{\tau}$,
\begin{align}\label{2.2}
|f_{\tau}(v)|\leq &C\Big\{
\tau^{1/2}
\sum_{i=1}^{k}|\nabla K(P_i)|+\tau+\tau|\log \tau|
\Big\}\|v\|_{\sigma}\notag\\
\leq &C\tau|\log \tau|\|v\|_{\sigma},
\end{align}
this,  combining \eqref{1.47} and \eqref{1.48}, we complete the proof.
\end{proof}

\begin{proposition}\label{lem5}
Under the assumptions of Theorem \ref{thm3},
in addition that $\Sigma_{\tau}(\overline{P}_{1},\cdots,\overline{P}_{k})$ is as in \eqref{1.75}.
Then for any $(\alpha, t,P, v)\in \Sigma_{\tau}(\overline{P}_{1},\cdots,\overline{P}_{k}),$  there exists
$V_{\alpha_i}(\tau,\alpha,t,P,v)$ such that
\begin{align*}
\frac{\partial}{\partial \alpha_{i}} I_{\tau}\Big(\sum_{j=1}^{k} \alpha_{j} \delta_{P_{j}, t_{j}}+v\Big)
=-|\mathbb{S}^3|\beta_{i}+V_{\alpha_i}(\tau,\alpha,t,P,v),
\end{align*}
where $\beta_{i}:=\alpha_{i}-1/K(P_{i}),$  $i=1,\cdots,k.$
Furthermore, let $\overline{v}$ be as in Proposition \ref{prop1}, then
\begin{align*}
\frac{\partial}{\partial \alpha_{i}} I_{\tau}\Big(\sum_{j=1}^{k} \alpha_{j} \delta_{P_{j}, t_{j}}+\overline{v}\Big)
=-|\mathbb{S}^3|\beta_i+O(|\beta|^2+\tau|\log \tau|).
\end{align*}
\end{proposition}

\begin{proof}

From \eqref{p1p1}, \eqref{a.10}, Lemma \ref{lem1}, Lemma \ref{lem2},
and \eqref{1-tau}, we have
\begin{align*}
&\frac{\partial}{\partial \alpha_{i}} I_{\tau}\Big(
\sum_{j=1}^{k} \alpha_{j} \delta_{P_{j}, t_{j}}+v\Big)
\nonumber \\
=&
\alpha_{i}\int_{\mathbb{S}^3}\delta_{P_i,t_i}^3
+\frac{1}{2}\sum_{j\ne i}\alpha_{j}\int_{\mathbb{S}^3}\delta_{P_j,t_j}^2\delta_{P_i,t_i}
\\
&\quad
-\int_{\mathbb{S}^{3}} K\Big|\sum_{j=1}^{k} \alpha_{j} \delta_{P_{j}, t_{j}}
+v\Big|^{1-\tau}\Big(\sum_{j=1}^{k} \alpha_{j} \delta_{P_{j}, t_{j}}+v\Big) \delta_{P_{i}, t_{i}}\nonumber\\
=
&|\mathbb{S}^3|\alpha_{i}
-\int_{\mathbb{S}^3}K\Big(\sum_{j=1}^{k}\alpha_{j}^{2-\tau}
\delta_{P_j,t_j}^{2-\tau}\Big)\delta_{P_i,t_i}
-(2-\tau)
\int_{\mathbb{S}^3}K\Big(\alpha_{i}^{1-\tau}\delta_{P_i,t_i}^{1-\tau}\Big)\delta_{P_i,t_i}v\\
&\quad
+O(\tau|\log\tau|)+O(\tau)+O(\|v\|_{\sigma}^{2-\tau}).
\end{align*}
By using \eqref{p-p1}, we obtain
\begin{align}\label{1.77}
\int_{\mathbb{S}^3}K\alpha_{i}^{2}\delta_{P_i,t_i}^{3-\tau}=&
\int_{\mathbb{S}^3}K(P_i)\alpha_{i}^{2}\delta_{P_i,t_i}^{3-\tau}
-\int_{\mathbb{S}^3}({K}(P)-K(P_i))\alpha_{i}^{2}\delta_{P_i,t_i}^{3-\tau}\notag\\
=&
\int_{\mathbb{S}^3}K(P_i)\alpha_{i}^{2}\delta_{P_i,t_i}^{3-\tau}
+O(\tau).
\end{align}
Similarly, by \eqref{1.12}, \eqref{a.9}, \eqref{stau}, and
 \eqref{p-p1}, we have
\begin{align}\label{1.78}
&\int_{\mathbb{S}^3}
K\alpha_{i}\delta_{P_i,t_i}^{2-\tau}v\notag\\
=&
\int_{\mathbb{S}^3}K(P_i) \alpha_{i}\delta_{P_i,t_i}^{2}v
+\int_{\mathbb{S}^3}(K(P)-K(P_i))\alpha_{i}\delta_{P_i,t_i}^{2}v
+O\left( \tau|\log \tau|\|v\|_{{\sigma}}\right)\notag\\
 =&O(\tau|\log\tau|)+O(\|v\|_{{\sigma}}^{2}).
\end{align}
 It follows from the fact $|\alpha_i^{2-\tau}-\alpha_{i}^{2}|=O(\tau),$ \eqref{1.77}, \eqref{1.78},
 \eqref{a.3}, and \eqref{a.7} that
\begin{align*}
&\frac{\partial}{\partial \alpha_{i}} I_{\tau}\Big(
\sum_{j=1}^{k} \alpha_{j} \delta_{P_{j}, t_{j}}+v\Big)\nonumber \\
=& |\mathbb{S}^3|\alpha_{i}
-\int_{\mathbb{S}^{3}} K\alpha_{i}^{2}\delta_{P_i,t_i}^{3-\tau}
-2\int_{\mathbb{S}^3}K\alpha_i\delta_{P_i,t_i}^{2-\tau}v+O(\tau|\log\tau|)+O(\|v\|_{\sigma}^{2-\tau})
\\
=&-\beta_{i}\int_{\mathbb{S}^3}\delta_{P_i,t_i}^3
+O(|\beta|^2)
+O(\tau|\log\tau|)+O(\|v\|_{\sigma}^{2-\tau})\\
=&-|\mathbb{S}^3|\beta_{i}
+O(|\beta|^2)
+O(\tau|\log\tau|)+O(\|v\|_{\sigma}^{2-\tau}).
\end{align*}
Hence
\begin{align}
\frac{\partial}{\partial \alpha_{i}} I_{\tau}\Big(\sum_{j=1}^{k} \alpha_{j} \delta_{P_{j}, t_{j}}+v\Big)
=-|\mathbb{S}^3|\beta_{i}+V_{\alpha_i}(\tau,\alpha,t,P,v),
\end{align}
where
$$V_{\alpha_i}(\tau,\alpha,t,P,v)=O(|\beta|^2)
+O(\tau|\log\tau|)
+O(\|v\|_{\sigma}^{2-\tau}).$$
Combining with Proposition \ref{prop1}, we get
\begin{align}\label{1.18}
\frac{\partial}{\partial \alpha_{i}} I_{\tau}\Big(\sum_{j=1}^{k} \alpha_{j} \delta_{P_{j}, t_{j}}+\overline{v}\Big)
=&-|\mathbb{S}^3|\beta_i+V_{\alpha_i}(\tau,\alpha,t,P,\overline{v})\notag\\
=&-|\mathbb{S}^3|\beta_i+O(|\beta|^2+\tau|\log \tau|).
\end{align}
Proposition \ref{lem5} follows from the above.
\end{proof}

\begin{proposition}\label{lem6}
Under the assumptions of  Theorem \ref{thm3},
in addition that $\Sigma_{\tau}(\overline{P}_{1},\cdots,\overline{P}_{k})$ is as in \eqref{1.75}.
Then for any $(\alpha, t,P, v)\in \Sigma_{\tau}(\overline{P}_{1},\cdots,\overline{P}_{k}),$  there exists
$V_{t_i}(\tau,\alpha, t,P,v)$ such that
\begin{align*}
&\frac{\partial}{\partial t_{i}}I_{\tau}\Big(\sum_{j=1}^{k} \alpha_{j}
\delta_{P_{j}, t_{j}}+v\Big)\notag\\
=&\frac{\Gamma_{3}}{K(P_i)^2}\frac{\tau}{t_{i}}
+\frac{\Gamma_{4}\Delta_{g_0}K(P_{i})}{K(P_i)^{3}}\frac{1}{t_i^3}
+\sum_{j\ne i}\frac{\Gamma_{5}G_{P_{i}}(P_{j})}{K(P_i)K(P_j)}\frac{1}{t_i^2t_j}
+V_{t_i}(\tau,\alpha, t,P,v),
\end{align*}
where
\begin{align*}
V_{t_i}(\tau,\alpha, t,P,v)=
O(|\beta|\tau^{3/2})
+
O(\tau \|v\|_{\sigma})+O(\tau^{1/2}\|v\|_{\sigma}^{2-\tau})
+O(\tau^{3/2}|\log \tau|),
\end{align*}
$\Gamma_{3},$ $\Gamma_{4},$ $\Gamma_{5}$ are positive constants,
 and $G_{P_{i}}(P_{j})$ is as in \eqref{gf}.
\end{proposition}

\begin{proof}
Using \eqref{1.26},  Lemma \ref{lem1}, \eqref{part11}, H\"{o}lder inequality, and Sobolev embedding,
 we have,
\begin{align}\label{1.31}
&\frac{\partial}{\partial t_{i}}I_{\tau}\Big(\sum_{j=1}^{k}
 \alpha_{j} \delta_{P_{j}, t_{j}}+v\Big)\notag\\
=&\frac{1}{2}
\sum_{j\ne i}\alpha_{i}\alpha_{j}\frac{\partial }{\partial t_i}
\int_{\mathbb{S}^3}\delta_{P_j,t_j}\delta_{P_i, t_i}^2
-\int_{\mathbb{S}^3}K\Big( \sum_{j=1}^{k}\alpha_j\delta_{P_j,t_j}\Big)^{2-\tau}\alpha_{i}\frac{\partial\delta_{P_i, t_i}}{\partial t_{i}}\notag\\
&\quad
-(2-\tau)\int_{\mathbb{S}^3}K\Big(\sum_{j=1}^{k}\alpha_{j}\delta_{P_j,t_j}\Big)^{1-\tau}v
\alpha_{i}\frac{\partial \delta_{P_i,t_i}}{\partial t_i}
+O\Big(\|v\|_{\sigma}^{2-\tau}\Big\|\frac{\partial \delta_{P_i,t_i}}{\partial t_i}\Big\|_{\sigma}\Big
).
\end{align}
By \eqref{1.12},  we have
\begin{align}\label{1.55}
\int_{\mathbb{S}^3}
\delta_{P_i,t_i}
\frac{\partial \delta_{P_i,t_i}}{\partial t_i} v
=&\frac{1}{2}\int_{\mathbb{S}^3}v\frac{\partial}{\partial t_i}(P_{\sigma} \delta_{P_i,t_i})\notag\\
=&\frac{1}{2}\frac{\partial}{\partial t_i}\langle v,\delta_{P_i,t_i}\rangle\notag\\
=&\frac{1}{2}\Big\langle v, \frac{\partial \delta_{P_i,t_i}}{\partial t_i}\Big\rangle\notag\\
=&0.
\end{align}
It follows from \eqref{1.55}, \eqref{a.9}, \eqref{part11}, and \eqref{p-p1} that
\begin{align*}
&\Big|\int_{\mathbb{S}^3} K\delta_{P_i,t_i}^{1-\tau}
\frac{\partial \delta_{P_i,t_i}}{\partial t_i} v\Big|\\
=&
\Big| \int_{\mathbb{S}^3}(K-K(P_i))\delta_{P_i,t_i}
\frac{\partial \delta_{P_i,t_i}}{\partial t_i}v
+\int_{\mathbb{S}^3}K(\delta_{P_i,t_i}^{1-\tau}-\delta_{P_i,t_i})
\frac{\partial \delta_{P_i,t_i}}{\partial t_i} v
 \Big|\\
 \leq&
C \tau^{1/2}|\log \tau|
 \int_{\mathbb{S}^3}|P-P_i|\delta_{P_i,t_i}
 \frac{\partial \delta_{P_i,t_i}}{\partial t_i}
 v
 +O\Big(\|\delta_{P_i,t_i}^{1-\tau}-\delta_{P_i,t_i}\|_{L^{3}(\mathbb{S}^3)}
 \Big\|\frac{\partial \delta_{P_i,t_i}}{\partial t_i}\Big\|_{\sigma}
 \|v\|_{\sigma}\Big)\\
 \leq&
 (\tau^{1/2}|\log \tau|)
 O\Big(\big\||P-P_i|\delta_{P_i,t_i}\big\|_{L^{3}(\mathbb{S}^3)}
 \Big\|\frac{\partial \delta_{P_i,t_i}}{\partial t_i}\Big\|_{\sigma}
 \|v\|_{\sigma}\Big)+O\big(\tau^{3/2}
 \|v\|_{\sigma}\big)\\
 \leq& C\tau\|v\|_{\sigma},
\end{align*}
this, and \eqref{a.11} yields
\begin{align}
&\Big|
\int_{\mathbb{S}^3}K\Big(\sum_{j=1}^{k}\delta_{P_j,t_j}\Big)^{1-\tau}
\frac{\partial \delta_{P_i,t_i}}{\partial t_i}v
\Big|\notag\\
\leq &
\Big|
\int_{\mathbb{S}^3}K\alpha_{i}^{1-\tau}
\delta_{P_i,t_i}^{1-\tau}\frac{\partial \delta_{P_i,t_i}}{\partial t_i}v
\Big|
+C\sum_{j\ne i} \int_{\mathbb{S}^3}\delta_{P_j,t_j}^{1-\tau}
\Big|\frac{\partial \delta_{P_i,t_i}}{\partial t_i}\Big||v|\notag\\
\leq& C\tau\|v\|_{\sigma}+
O\Big( \sum_{j\ne i}\Big\|\delta_{P_j,t_j}^{1-\tau}
\frac{\partial \delta_{P_i,t_i}}{\partial t_i}\Big\|_{L^{3/2}(\mathbb{S}^3)}\|v\|_{\sigma}
\Big)
\notag\\
\leq& C\tau\|v\|_{\sigma}.\label{1.27}
\end{align}
Using \eqref{1.27} and Lemma \ref{lem2}, we obtain
\begin{align*}
&\frac{\partial}{\partial t_{i}}I_{\tau}
\Big(\sum_{j=1}^{k} \alpha_{j} \delta_{P_{j}, t_{j}}+v\Big)\\
=&
\frac{1}{2}
\sum_{j\ne i}\alpha_{i}\alpha_{j}\frac{\partial}{\partial t_i}\int_{\mathbb{S}^3}\delta_{P_j,t_j}\delta_{P_i, t_i}^2
-\int_{\mathbb{S}^3}K\Big( \sum_{j=1}^{k}\alpha_j\delta_{P_j,t_j}\Big)^{2-\tau}
\alpha_{i}\frac{\partial\delta_{P_i, t_i}}{\partial t_{i}}\\
&\quad +O(\tau\|v\|_{\sigma})
+O(\tau^{1/2}\|v\|_{\sigma}^{2-\tau})\\
=&
\frac{1}{2}
\sum_{j\ne i}\alpha_{i}\alpha_{j}\frac{\partial}{\partial t_i}
\int_{\mathbb{S}^3}\delta_{P_j,t_j}\delta_{P_i,t_i}^2
-\int_{\mathbb{S}^3}K\alpha_{i}^{3}\delta_{P_i,t_i}^{2-\tau}
\frac{\partial \delta_{P_i,t_i}}{\partial t_i}
 \\
&\quad
-\sum_{j\ne i}\int_{\mathbb{S}^3}
\alpha_{i}K \alpha_{j}^{2}\delta_{P_j,t_j}^{2-\tau}
\frac{\partial \delta_{P_i,t_i}}{\partial t_i}\\
&\quad
+O(\tau \|v\|_{\sigma})+O(\tau^{1/2}\|v\|_{\sigma}^{2-\tau})
+O(\tau^{3/2}|\log \tau|).
\end{align*}
We have used the following facts. By \eqref{1-tau}, \eqref{a.11}, \eqref{part11},
and  \eqref{a.6}, we have
\begin{align*}
&\sum_{j\ne \ell}\int_{\mathbb{S}^3} \delta_{P_j,t_j}^{1-\tau}\delta_{P_{\ell}, t_{\ell}}
\frac{\partial \delta_{P_i, t_i}}{\partial t_i}\\
=&
\int_{\mathbb{S}^3}\Big( \sum_{j=i,\ell\ne i}
\delta_{P_i,t_i}^{1-\tau}\delta_{P_{\ell}, t_{\ell}}\frac{\partial \delta_{P_i, t_i}}{\partial t_i}
+\sum_{j\ne i,\ell=i}\delta_{P_j,t_j}^{1-\tau}\delta_{P_i,t_i}\frac{\partial \delta_{P_i, t_i}}{\partial t_i}
+\sum_{j\ne\ell\ne i} \delta_{P_j,t_j}^{1-\tau}\delta_{P_{\ell}, t_{\ell}}
\frac{\partial \delta_{P_i, t_i}}{\partial t_i} \Big)\\
=&
O\Big(\sum_{\ell\ne i}\frac{\partial }{\partial t_i}\int_{\mathbb{S}^3}\delta_{P_{\ell},t_{\ell}}
\delta_{P_i,t_i}^{2-\tau}\Big)
+O\Big(\sum_{j\ne i}\Big\| \delta_{P_j,t_j}^{1-\tau}
\frac{\partial \delta_{P_i,t_i}}{\partial t_i}\Big\|_{L^{3/2}(\mathbb{S}^3)}\Big)\\
\quad&+
O\Big(\sum_{j\ne \ell\ne i}\left\|\delta_{P_j,t_j}^{1-\tau}\delta_{P_{\ell},t_{\ell}}\right\|_{L^{3/2}(\mathbb{S}^3)}
\Big\|\frac{\partial \delta_{P_i,t_i}}{\partial t_i}\Big\|_{\sigma}\Big)\\
=&O(\tau^{3/2}|\log \tau|).
\end{align*}
Using \eqref{partp12} and  \eqref{partk-kpj}, we have
\begin{align}\label{cald}
&\int_{\mathbb{S}^3}K\delta_{P_{j}, t_{j}}^{2-\tau}\frac{\partial \delta_{P_i, t_i}}{\partial t_i}\notag\\
=&
\frac{\partial}{\partial t_i} \int_{\mathbb{S}^3}K \delta_{P_j, t_j}^{2-\tau} \delta_{P_i, t_i}\notag\\
=&
K(P_j)\frac{\partial}{\partial t_i}\int_{\mathbb{S}^3} \delta_{P_j, t_j}^{2} \delta_{P_i, t_i}
+O(\tau^{5/2} |\log \tau|)
+O\Big( \frac{\partial}{\partial t_i}\int_{\mathbb{S}^3}(K(P)-K(P_j))\delta_{P_j, t_j}^{2-\tau} \delta_{P_i, t_i}  \Big)\notag\\
=&
K(P_j)\frac{\partial}{\partial t_i}\int_{\mathbb{S}^3}
\delta_{P_j, t_j}^{2} \delta_{P_i, t_i}+
O(\tau^{5/2} |\log \tau|)
+O(\tau^2).
\end{align}
By \eqref{a.8}, we have
\begin{align}
&-\int_{\mathbb{S}^3}K \alpha_{i}^{3}\delta_{P_{i}, t_i}^{2-\tau}\frac{\partial \delta_{P_i, t_i}}{\partial t_i}
\notag\\
=&
-\frac{1}{3-\tau}\alpha_{i}^3K(P_i)\frac{\partial }{\partial t_i}\int_{\mathbb{S}^3}\delta_{P_i,t_i}^{3-\tau}\notag\\
&\quad
-\frac{2}{3(3-\tau)}\Delta_{g_0} {K}(P_i)\frac{\partial }{\partial t_i}
\int_{\mathbb{S}^3}
|P-P_i|^2\alpha_{i}^3
\delta_{P_i,t_i}^{3-\tau}+O(\tau^2).\label{1.25}
\end{align}

Let
\begin{align}\label{calE}
\mathscr{E}=O(\tau \|v\|_{\sigma})+O(\tau^{1/2}\|v\|_{\sigma}^{2-\tau})
+O(\tau^{3/2}|\log \tau|),
\end{align}
then, by \eqref{part3-tau}, \eqref{a.5}, and \eqref{1.25}, we have
\begin{align*}
&\frac{\partial}{\partial t_{i}}I_{\tau}\Big(\sum_{j=1}^{k} \alpha_{j} \delta_{P_{j}, t_{j}}+v\Big)\\
=&\frac{1}{2}
\sum_{j\ne i}\alpha_{i}\alpha_{j}\frac{\partial}{\partial t_i}\int_{\mathbb{S}^3}\delta_{P_j,t_j}\delta_{P_i, t_i}^2
-\int_{\mathbb{S}^3}K \alpha_{i}^{3}\delta_{P_{i}, t_i}^{2-\tau}\frac{\partial \delta_{P_i, t_i}}{\partial t_i}\\
\quad&
-\alpha_i\sum_{j\ne i}K(P_j)\alpha_{j}^{2}
\frac{\partial}{\partial t_i}\int_{\mathbb{S}^3} \delta_{P_{j}, t_j}^{2} \delta_{P_i, t_i}
+\mathscr{E}\\
=&
\sum_{j\neq i}\Big\{\frac{1}{2}\alpha_{i}\alpha_{j}-\alpha_i\alpha_j^2K(P_{j})\Big\}
\frac{\partial}{\partial t_i}\int_{\mathbb{S}^3}\delta_{P_j,t_j}\delta_{P_i, t_i}^2\\
\quad&
-\frac{1}{3-\tau}\alpha_i^3K(P_i)\frac{\partial}{\partial t_i}\int_{\mathbb{S}^3}
\delta_{P_i, t_i}^{3-\tau}\\
\quad&
-\frac{2}{3(3-\tau)}\alpha_i^{3}\Delta_{g_0}K(P_i)
\frac{\partial}{\partial t_i}\int_{\mathbb{S}^3}|P-P_i|^2\delta_{P_i,t_i}^{3-\tau}
+O(\tau^{5/2})
+\mathscr{E}\\
=&
-\frac{1}{2}\sum_{j\ne i}\frac{1}{K(P_{i})K(P_{j})}\frac{\partial}{\partial t_i}\int_{\mathbb{S}^3}\delta_{P_j,t_j}\delta_{P_i, t_i}^2
-\frac{1}{3}\frac{1}{K(P_i)^2}\frac{\partial}{\partial t_i}
 \int_{\mathbb{S}^3}\delta_{P_i,t_i}^{3-\tau}\\
\quad&
-\frac{2}{9}\frac{\Delta_{g_0}K(P_{i})}{K(P_i)^3}
\frac{\partial}{\partial t_i}\int_{\mathbb{S}^3}|P-P_i|^2\delta_{P_i,t_i}^{3-\tau}
+\mathscr{E}+O(|\beta| \tau^{3/2}).
\end{align*}

It follows from \eqref{a.1}, \eqref{part3-tau}, and \eqref{a.5} that
\begin{align}\label{1.32}
&\frac{\partial}{\partial t_{i}}I_{\tau}\Big(\sum_{j=1}^{k} \alpha_{j}
\delta_{P_{j}, t_{j}}+v\Big)\notag\\
=&\frac{\Gamma_{3}}{K(P_i)^2}\frac{\tau}{t_{i}}
+\frac{\Gamma_{4}\Delta_{g_0}K(P_{i})}{K(P_i)^{3}}\frac{1}{t_i^3}
\notag\\
\quad&
+\sum_{j\ne i}\frac{\Gamma_{5}G_{P_{i}}(P_{j})}{K(P_i)K(P_j)}\frac{1}{t_i^2t_j}
+V_{t_i}(\tau,\alpha, t,P,v),
\end{align}
where
\begin{align*}
&V_{t_i}(\tau,\alpha, t,P,v)=
O(|\beta|\tau^{3/2})
+
O(\tau \|v\|_{\sigma})+O(\tau^{1/2}\|v\|_{\sigma}^{2-\tau})
+O(\tau^{3/2}|\log \tau|),\notag\\
&\Gamma_{3}=\frac{4}{3}\pi|\mathbb{S}^2|,\quad
\Gamma_{4}=\frac{2}{3}\pi |\mathbb{S}^2|,\quad
\Gamma_{5}=2\pi |\mathbb{S}^2|.
\end{align*}
Proposition \ref{lem6} follows from the above.
\end{proof}

\begin{proposition}\label{lem7}
Under the assumptions of Theorem \ref{thm3},
in addition that $\Sigma_{\tau}(\overline{P}_{1},\cdots,\overline{P}_{k})$ is as in \eqref{1.75}.
Then for any $(\alpha, t,P, v)\in \Sigma_{\tau}(\overline{P}_{1},\cdots,\overline{P}_{k}),$  there exists
 a constant $\nu_{1}>0$ independent of $\tau$ and a vector $V_{P_i}(\tau,\alpha,t,P,v),$ such that
\begin{align*}
&\frac{\partial}{\partial P_{i}}
I_{\tau}\Big(\sum_{j=1}^{k} \alpha_{j} \delta_{P_{j}, t_{j}}+v\Big)
=-\Gamma_6\nabla_{g_0}{K}(P_i)
+V_{P_i}(\tau,\alpha,t,P,v),
\end{align*}
where $\Gamma_{6}\geq\nu_{1}>0$ is a constant,
and
\begin{align*}
V_{P_i}(\tau,\alpha,t,P,v)=O(\tau^{1/2})
+O( \|v\|_{\sigma})
+O( \tau^{-1/2}\|v\|_{\sigma}^{2-\tau}).
\end{align*}
\end{proposition}

\begin{proof}

Using Lemma \ref{lem1}, we have
\begin{align}\label{1.33}
&\frac{\partial}{\partial P_{i}}
I_{\tau}\Big(\sum_{j=1}^{k} \alpha_{j} \delta_{P_{j}, t_{j}}+v\Big)\notag\\
=&
\frac{1}{2}\sum_{j\ne i}
\alp_i\alp_j\frac{\partial}{\partial P_i}\int_{\mathbb{S}^3}
\alp_{j}
\delta_{P_{j},t_{j}}^{2}\delta_{P_i,t_i}
\notag\\
\quad&
-\int_{\mathbb{S}^3}K\Big|\sum_{j=1}^{k}\alp_j\delta_{P_j,t_j}+v\Big|^{1-\tau}\Big( \sum_{j=1}^{k}\alp_j\delta_{P_j,t_j}+v\Big)
\alp_i\frac{\partial\delta_{P_i, t_i}}{\partial P_{i}}\notag\\
=&\frac{1}{2}\sum_{j\ne i}
\alp_i\alp_j\frac{\partial}{\partial P_i}\int_{\mathbb{S}^3}
\alp_{j}
\delta_{P_{j},t_{j}}^{2}\delta_{P_i,t_i}\notag\\
&\quad-\int_{\mathbb{S}^3}K\Big(\sum_{j=1}^{k}\alp_{j}\delta_{P_j,t_j}\Big)^{2-\tau}
\alp_{i}\frac{\partial \delta_{P_i,t_i}}{\partial P_i}\notag\\
&\quad
-(2-\tau)\int_{\mathbb{S}^3}
K\Big(\sum_{j=1}^{k}\alpha_{j}\delta_{P_{j},t_{j}}\Big)^{1-\tau}v
\alpha_{i}\frac{\partial \delta_{P_i,t_i}}{\partial P_i}\notag\\
&\quad
+O\Big( \|v\|_{\sigma}^{2-\tau}
\Big\|\frac{\partial \delta_{P_i,t_i}}{\partial P_i}\Big\|_{\sigma}\Big).
\end{align}

By \eqref{1.12}, we have
\begin{align}\label{1.28}
\int_{\mathbb{S}^3}\delta_{P_i,t_i}\frac{\partial \delta_{P_i,t_i}}{\partial P_i} v
=\frac{1}{2}\frac{\partial }{\partial P_i}
\int_{\mathbb{S}^3}\delta_{P_i,t_i}^2 v=
\frac{1}{2}\big\langle \frac{\partial \delta_{P_i,t_i}}{\partial P_i}, v \big\rangle =0.
\end{align}

It follows  from \eqref{a.12}, \eqref{a.13}, \eqref{1.28}, and \eqref{a.9} that
\begin{align}
&\int_{\mathbb{S}^3}
K\Big(\sum_{j=1}^{k}\alpha_{j}\delta_{P_{j},t_{j}}\Big)^{1-\tau}v
\alpha_{i}\frac{\partial \delta_{P_i,t_i}}{\partial P_i}
\notag\\
=&\int_{\mathbb{S}^3}
K(\alpha_{i}^{1-\tau}\delta_{P_i,t_i}^{1-\tau})\alpha_{i}
\frac{\partial \delta_{P_i,t_i}}{\partial P_{i}}v
+O\Big(\sum_{j\ne i}\int_{\mathbb{S}^3}
\delta_{P_j,t_j}^{1-\tau}
\Big|\frac{\partial \delta_{P_i,t_i}}{\partial P_i}\Big|
|v|
\Big)
\notag\\
=&K(P_i)\alpha_{i}^{2-\tau}\int_{\mathbb{S}^3}
\delta_{P_i,t_i}^{1-\tau}
\frac{\partial \delta_{P_i,t_i}}{\partial P_{i}}v
+O(\|v\|_{\sigma})\notag\\
=&O(\|v\|_{\sigma}).\label{1.29}
\end{align}

Then  Lemma \ref{lem2}, \eqref{1.29}, \eqref{p1p1}, and \eqref{a.7}
 yields
\begin{align*}
&\frac{\partial}{\partial P_{i}}
I_{\tau}\Big(\sum_{j=1}^{k} \alpha_{j} \delta_{P_{j}, t_{j}}+v\Big)\\
=&\frac{1}{2}\sum_{j\ne i}
\alp_i\alp_j\frac{\partial}{\partial P_i}\int_{\mathbb{S}^3}
\delta_{P_{j},t_{j}}^{2}\delta_{P_i,t_i}\\
&\quad
-\alp_i\int_{\mathbb{S}^3}K(\alp_i\delta_{P_i,t_i})^{2-\tau}
\frac{\partial \delta_{P_i,t_i}}{\partial P_i}
-\sum_{j\neq i}\alp_i\int_{\mathbb{S}^3}
K(\alp_j\delta_{P_j,t_j})^{2-\tau}
\frac{\partial \delta_{P_i,t_i}}{\partial P_i}
\\
&\quad
+O( \|v\|_{\sigma})
+O\left(\tau^{-1/2}\|v\|_{\sigma}^{2-\tau}
\right)\\
=&
O\Big( \sum_{j\neq i}\frac{\partial}{\partial P_i}
\int_{\mathbb{S}^3}\delta_{P_j,t_j}^2\delta_{P_i,t_i} \Big)\\
&\quad
-\alp_i^3\int_{\mathbb{S}^3}K\delta_{P_i,t_i}^{2-\tau}
\frac{\partial \delta_{P_i,t_i}}{\partial P_i}
+O\Big(\tau\Big|\int_{\mathbb{S}^3}\delta_{P_i,t_i}^{2-\tau}
\frac{\partial \delta_{P_i,t_i}}{\partial P_i}\Big|\Big)\\
&\quad+O\Big(\sum_{j\neq i}\Big|\int_{\mathbb{S}^3}\delta_{P_{j},t_j}^{2-\tau}
\frac{\partial \delta_{P_i,t_i}}{\partial P_{i}}\Big| \Big)
+O( \|v\|_{\sigma})
+O(\tau^{-1/2}\|v\|_{\sigma}^{2-\tau}
)\\
=&
-\alp_i^3\int_{\mathbb{S}^3}K\delta_{P_i,t_i}^{2-\tau}\frac{\partial \delta_{P_i,t_i}}{\partial P_i}
+O(\tau^{1/2})
+O( \|v\|_{\sigma})
+O(\tau^{-1/2}\|v\|_{\sigma}^{2-\tau}).
\end{align*}
Thus, by \eqref{a.7},
\begin{align*}
&\frac{\partial}{\partial P_{i}}
I_{\tau}\Big(\sum_{j=1}^{k} \alpha_{j} \delta_{P_{j}, t_{j}}+v\Big)\\
=&
-\alp_i^3\int_{\mathbb{S}^3}(K(P)-K(P_i))\delta_{P_i,t_i}^{2-\tau}
\frac{\partial \delta_{P_i,t_i}}{\partial P_i}
-\alp_i^3\int_{\mathbb{S}^3}K(P_i)\delta_{P_i,t_i}^{2-\tau}
\frac{\partial \delta_{P_i,t_i}}{\partial P_i}\\
&\quad+O(\tau^{1/2})
+O\left( \|v\|_{\sigma}\right)
+O\left( \tau^{-1/2}\|v\|_{\sigma}^{2-\tau}\right)\\
=&-\Gamma_6\nabla_{g_0}{K}(P_i)
+V_{P_i}(\tau,\alpha,t,P,v),
\end{align*}
where
\begin{align}\label{1.56}
\Gamma_6(\tau,\alpha, t,P,v)\geq \nu_{1}>0 \quad\hbox{with}\,
\nu_1 \hbox{ independent  of}\, \tau,
\end{align}
and
\begin{align}\label{1.34}
V_{P_i}(\tau,\alpha,t,P,v)=O(\tau^{1/2})
+O\left( \|v\|_{\sigma}\right)
+O\left( \tau^{-1/2}\|v\|_{\sigma}^{2-\tau}\right).
\end{align}

The existence of $\nu_{1}$ is proved below.
In fact, let $P_{i}$ be the south pole and make a stereographic projection $F$ to the
equatorial plane of $\mathbb{S}^3$ with $y=(y^{(1)},y^{(2)},y^{(3)})$
as the stereographic projection coordinates,
let $\widetilde{K}=K(F(y))$
and  $|J_{F}|:=(2/(1+|y|^{2}))^{3}.$
Then  we have $F(0)=P_{i}$ and
\begin{align*}
&\alp_i^3\int_{\mathbb{S}^3}(K(P)-K(P_i))\delta_{P_i,t_i}^{2-\tau}
\frac{\partial \delta_{P_i,t_i}}{\partial P_i}\\
=&
\alp_{i}^3\int_{\mathbb{R}^3}
t_{i}y(\wdt{K}(y)-\wdt{K}(0))
\omega_{0,t_{i}}^{4}
(|J_{F}|^{1/3}\omega_{0,t_{i}}^{-1})^{\tau}
\\
=&:\mathcal{L}=(\mathcal{L}^{(1)},\mathcal{L}^{(2)},\mathcal{L}^{(3)}),
\end{align*}
where
$\omega_{0,t_{i}}:=2t_{i}/(1+t_{i}^2|y|^2)$ is the solution of
$$
(-\Delta)^{1/2} \omega_{0,t_{i}}
= \omega_{0,t_{i}}^2 \quad \text { on }\, \mathbb{R}^{3}.
$$
For $j=1, 2, 3$, we have
\begin{align*}
\mathcal{L}^{(j)}
=&-\alp_{i}^3\int_{\mathbb{R}^3}
t_{i}y^{(j)}(\wdt{K}(y)-\wdt{K}(0))
\omega_{0,t_{i}}^{4}
(|J_{F}|^{1/3}\omega_{0,t_{i}}^{-1})^{\tau}
\\
=&
\alp_{i}^3\int_{\mathbb{R}^3}
t_{i}y^{(j)}(\nabla\wdt{K}(0)\cdot y+O(|y|^2))
\omega_{0,t_{i}}^{4}
(|J_{F}|^{1/3}\omega_{0,t_{i}}^{-1})^{\tau}
\\
=&\frac{1}{3}\alp_{i}^3\frac{\partial \wdt{K}}{\partial y^{(j)}}(0)\int_{\mathbb{R}^3}
t_{i}|y|^2\omega_{0,y_{i}}^{4}
(|J_{F}|^{1/3}\omega_{0,t_{i}}^{-1})^{\tau},
\end{align*}
thus,
\begin{align*}
\mathcal{L}=&\nabla\widetilde{K}(0)
\Big\{\frac{\alp_{i}^3}{3}
\int_{\mathbb{R}^3}t_{i}|y|^2
\omega_{0,t_{i}}^4
(|J_{F}|^{1/3}\omega_{0,t_{i}}^{-1})^{\tau}
+O(\tau^{1/2})\Big\}\\
=&\nabla_{g_0}K(P_i)
\frac{2\alp_{i}^3}{3}
\Big\{\int_{\mathbb{R}^3}t_{i}|y|^2
\omega_{0,t_{i}}^4
(|J_{F}|^{1/3}\omega_{0,t_{i}}^{-1})^{\tau}
+O(\tau^{1/2})\Big\}.
\end{align*}
It follows from
$
t_{i}^{-\tau}\leq(|J_{F}|^{1/3}\omega_{0,t_{i}}^{-1})^{\tau}\leq t_{i}^{\tau}
$
that
\begin{align*}
\int_{\mathbb{R}^3}
t_{i}|y|^2\omega_{0,t_{i}}^4
(|J_{F}|^{1/3}\omega_{0,t_{i}}^{-1})^{\tau}
\geq&t_{i}^{-\tau}
\int_{\mathbb{R}^3} t_{i}|y|^2\omega_{0,t_{i}}^4
\rightarrow
\int_{\mathbb{R}^3}\frac{|y|^2}{(1+|y|^2)^{3}},
\end{align*}
as $\tau \rightarrow 0.$
This ensures the existence of $\nu_{1}.$
We have proved Proposition \ref{lem7}.
\end{proof}

We now apply Propositions \ref{prop1}, \ref{lem5}, \ref{lem6}, \ref{lem7}
and construct a family of homotopy Id+compact operators
to obtain the degree-counting formula of the solutions to the subcritical equation \eqref{subequ} on
$\Sigma_{\tau}(\overline {P}_{1},\cdots,\overline{P}_{k}).$

\begin{proof}[Proof of Theorem \ref{thm3}]
Given $\tau>0$ and $K\in \mathscr{A}$, let $\mathscr{K}^{-}$ be as in \eqref{1.65}
and $\Sigma_{\tau}(\overline{P}_{1},\cdots,\overline{P}_{k})$ be as in \eqref{stau}
for the given $\overline{P}_{1},\cdots,\overline{P}_{k}\in \mathscr{K}^{-}.$

For $u=\sum_{i=1}^{k}\alp_{i}\delta_{P_i,t_i}+v\in \Sigma_{\tau}(\overline{P}_{1},
\cdots,\overline{P}_{k}),$  we have
\begin{align*}
T_{u}H^{\sigma}(\mathbb{S}^3)=E_{P,t}\bigoplus
 \mathrm{span}\{\delta_{P_i,t_i},\frac{\partial \delta_{P_i,t_i}}{\partial t_i},
 \frac{\partial \delta_{P_i,t_i}}{\partial P_i} \}.
\end{align*}
Since $I'_{\tau}(u)\in T_{u}H^{\sigma}(\mathbb{S}^3),$  there exist $\xi \in E_{P,t},$ $\eta\in \mathrm{span}\big\{\delta_{P_i,t_i},\frac{\partial \delta_{P_i,t_i}}{\partial t_i},
 \frac{\partial \delta_{P_i,t_i}}{\partial P_i} \big\} $ such that
\begin{align*}
I'_{\tau}(u)=\xi+\eta.
\end{align*}
From \eqref{1.66},  we  obtain, for all $\varphi\in E_{P,t}$,
\begin{align}\label{1.95}
\left\langle \xi, \varphi\right\rangle
=I'_{\tau}(u)\varphi=
f_{\tau}(\varphi)+2Q_{\tau}(v,\varphi)+
\left\langle V_{v}(\tau,\alpha,t,P,v),\varphi\right\rangle,
\end{align}
where $\|V_{v}(\tau,\alpha,t,P,v)\|_{\sigma}\leq C\|v\|_{\sigma}^{2-\tau}.$
Replacing  $\varphi$  by $v$ in \eqref{1.95} and using \eqref{1.68}, we have
\begin{align*}
\|\xi\|_{\sigma}\geq \delta_{0}\|v\|_{\sigma}-
\|{f}_{\tau}\|_{\sigma}-O(\|v\|_{\sigma}^{2-\tau})
\geq \frac{\delta_{0}}{2}\|v\|_{\sigma}-\|f_{\tau}\|,
\end{align*}
where $\delta_{0}$ is as in \eqref{1.68}.

Let $\beta=(\beta_{1},\cdots,\beta_{k}),$
$\beta_{i}=\alpha_{i}-1/K(P_{i})$ be as in Proposition \ref{lem5}, we define
\begin{align*}
\widehat{\Sigma}_{\tau}=
\Big\{ u=\sum_{i=1}^{k}\alpha_{i}\delta_{P_i,t_i}+v
\in \Sigma_{\tau}(\overline{P}_{1},\cdots,\overline{P}_{k}):\|v\|_{\sigma}<\tau|\log\tau|^3,\,
|\beta|<\tau|\log\tau|^2\Big\}.
\end{align*}
It follows from Proposition \ref{prop1} and \eqref{1.18} that
  $$I'_{\tau}(u)\ne 0, \quad\forall\, u\in \Sigma_{\tau}(\overline{P}_{1},\cdots,\overline{P}_{k})
 \backslash \widehat{\Sigma}_{\tau}.$$

For  $u=\sum_{i=1}^{k}\alpha_{i}\delta_{P_i,t_i}+v\in \widehat{\Sigma}_{\tau},$
by using \eqref{1.26}, we have
\begin{align}
\left\langle  \eta, \delta_{P_i,t_i}\right\rangle
 =&I_{\tau}'(u)\delta_{P_i,t_i}\notag\\
 =&
 \alpha_{i}\int_{\mathbb{S}^3}\delta_{P_i,t_i}^3
 +\frac{1}{2}\sum_{j\ne i}\alpha_{j}
\int_{\mathbb{S}^3}\delta_{P_i,t_i}^{2}\delta_{P_j,t_j}\notag \\
&\quad
-\int_{\mathbb{S}^3}
K\Big|\sum_{j=1}^{k}\alpha_j\delta_{P_j,t_j}+v\Big|^{1-\tau}
\Big(\sum_{j=1}^{k}\alpha_{j}\delta_{P_j,t_j}+v\Big)\delta_{P_i,t_i}\notag\\
=&
\frac{\partial }{\partial \alpha_{i}}I_{\tau}\Big(\sum_{j=1}^{k}\alpha_{j}\delta_{P_j,t_j}+v\Big)\notag\\
=&-|\mathbb{S}^3|\beta_{i}+V_{\alp_i}(\tau,\alp,t,P,v),\label{1.98}
 \end{align}
where $V_{\alp_i}$ satisfying
\begin{align*}
V_{\alpha_{i}}(\tau,\alp,t,P,v)= &
O(|\bt|^2)
+O(\tau|\log\tau|)
+O(\|v\|_{\sigma}^{2-\tau})\\
\leq&C\left(|\beta|^2+\tau|\log \tau|\right).
\end{align*}

It follows from \eqref{1.31} and \eqref{1.32} that
\begin{align}
\Big\langle \eta, \frac{\partial \delta_{P_i,t_i}}{\partial t_i}\Big\rangle
=&I'_{\tau}(u)\,\frac{\partial \delta_{P_i,t_i}}{\partial t_i}\notag\\
=&
\frac{1}{2}\sum_{j\ne i}\alpha_{j}
\frac{\partial }{\partial t_i}\int_{\mathbb{S}^3}
\delta_{P_j,t_j}^2
\delta_{P_i,t_i}
\notag\\
&\quad
-\int_{\mathbb{S}^3}
K\Big|\sum_{j=1}^{k}\alpha_j\delta_{P_j,t_j}+v\Big|^{1-\tau}
\Big(\sum_{j=1}^{k}\alpha_{j}\delta_{P_j,t_j}+v\Big)
\frac{\partial \delta_{P_i,t_i}}{\partial t_i}\notag\\
=&\frac{1}{\alpha_{i}}
\frac{\partial }{\partial t_i}
I_{\tau}\Big(\sum_{j=1}^{k}\alpha_{j}\delta_{P_j,t_j}+v\Big)\notag\\
=&\frac{1}{\alpha_{i}}
\Big\{
\frac{\Gamma_{3}}{K(P_i)^2}\frac{\tau}{t_{i}}
+\frac{\Gamma_{4}\Delta_{g_0}K(P_{i})}{K(P_i)^{3}}\frac{1}{t_i^3} \Big.\notag\\
&\phantom{=\;\;}\Big.
+\sum_{j\ne i}\frac{\Gamma_{5}G_{P_{i}}(P_{j})}{K(P_i)K(P_j)}
\frac{1}{t_i^2t_j}
+V_{t_i}(\tau,\alpha, t,P,v)
\Big\},\label{1.99}
\end{align}
where
$|V_{t_i}(\tau, \alpha, t,P,v)|=O(\tau^{3/2}|\log \tau|)$.

Applying \eqref{1.33} and \eqref{1.34}, we obtain
\begin{align}\label{2.1}
\Big\langle \eta, \frac{\partial \delta_{P_i,t_i}}{\partial P_{i}}\Big\rangle
=&I_{\tau}'(u)\frac{\partial \delta_{P_i,t_i}}{\partial P_{i}}\notag\\
=&\frac{1}{\alpha_{i}}
\frac{\partial }{\partial P_i}
I_{\tau}\Big(\sum_{i=1}^{k} \alpha_{i}\delta_{P_i,t_i}+v\Big)\notag\\
=&\frac{1}{\alpha_{i}}
\left\{
-\Gamma_{6}\nabla_{g_0}K(P_{i})
+V_{P_i}(\tau, \alpha, t,P, v)
\right\},
\end{align}
with $V_{P_i}$ satisfying $\left|V_{P_i}(\tau, \alpha, t, P,v)\right|\leq C\tau^{1/2}$.

Under the conditions  \eqref{1.95}--\eqref{2.1} stated above, we define a family of
operators on $\wdt{\Sigma}_{\tau}$ as follows:
for $u=\sum_{i=1}^{k}\alpha_{i}\delta_{P_i,t_i}+v\in \widehat{\Sigma}_{\tau}$ given
above,
$$X_{\theta}(u):=\xi_{\theta}(u)+\eta_{\theta}(u),\quad 0\leq\theta\leq 1,$$
where for any $\varphi\in E_{P,t},$
\begin{align}\label{xi}
\langle
\xi_{\theta},\varphi
\rangle:=&
\theta f_{\tau}(\varphi)+(1-\theta)\langle v,\phi\rangle
+2\theta Q_{\tau}(\varphi, v)
+\theta\langle V_{v}(\tau,\alpha, t,P,v),\varphi \rangle,
\end{align}
and
\begin{align}
\begin{aligned}\label{eta}
\langle \eta_{\theta}, \delta_{P_{i},t_i}  \rangle
:=&-2\pi |\mathbb{S}^3|
\Big\{ \alpha_{i}-\frac{\theta}{K(P_{i})}-\frac{(1-\theta)}{K(\overline{P}_{i})} \Big\}
+\theta V_{\alpha_{i}}(\tau, \alpha,t,P,v),\\
\Big\langle \eta_{\theta}, \frac{\partial \delta_{P_i,t_i}}{\partial t_i}  \Big\rangle
:=&\Big\{
\frac{1}{\alpha_{i}}+(1-\theta)
\Big\}
\Big\{
\frac{\Gamma_{3}}{K(P_i(\theta))^2}\frac{\tau}{t_{i}}
+\frac{\Gamma_{4}\Delta_{g_0}K(P_{i}(\theta))}{K(P_i(\theta))^{3}}\frac{1}{t_i^3} \Big.\\
&\phantom{=\;\;}\Big.
+\sum_{j\ne i}\frac{\Gamma_{5}G_{P_{i}(\theta)}(P_{j}(\theta))}{K(P_i(\theta))K(P_j(\theta))}
\frac{1}{t_i^2t_j}
\Big\}+\frac{\theta}{\alpha_{i}}V_{t_i}(\tau,\alpha, t,P,v),\\
\Big\langle
\eta_{\theta}, \frac{\partial \delta_{P_{i}, t_{i}}}{\partial P_{i}}
\Big\rangle
:=&-\Big\{
(1-\theta)+\frac{\theta}{\alpha_{i}}\Gamma_{6}
\Big\}
\nabla_{g_{0}}K(P_{i})+\frac{\theta}{\alpha_{i}}
V_{P_{i}}(\tau, \alpha, t,P, v),
\end{aligned}
\end{align}
where $P_{i}(\theta)$ is the short geodesic trajectory on $\mathbb{S}^3$ with $
P_{i}(0)=\overline{P}_{i},$
$P_{i}(1)=P_{i}.$

Obviously, $X_{1}=I_{\tau}'(u)=\xi+\eta.$  From Sobolev compact embedding theorem
and the explicit forms of $V_{v}, V_{\alpha_{i}}, V_{t_{i}}, V_{P_{i}}$, we conclude  that $I_{\tau}'(u)$ is of the form Id+compact on $\widehat{\Sigma}_{\tau}$.
Since $\Omega_{\varepsilon_{0}/2}$ in the definition
 of $\widehat{\Sigma}_{\tau}$ is a finite dimensional
 submanifold of $H^{\sigma}(\mathbb{S}^3),$ we easily obtain from \eqref{xi} and \eqref{eta}
 that $X_{\theta}$ $(0\leq \theta\leq 1)$ is the form Id+compact.
Furthermore, we have $X_{\theta}\ne 0 $  on $\partial \widehat{\Sigma}_{\tau},$
$\forall\, 0\leq \theta\leq 1.$
In fact, for a given $u=\sum_{i=1}^{k} \alpha_{i}
\delta_{P_{i},t_{i}}+v\in \partial \widehat{\Sigma}_{\tau},$
 we obtain $\xi\ne0$ by using \eqref{1.95} and \eqref{2.2}.
When $\theta=0,$ $\xi_{0}=v\ne 0.$
It follows from  \eqref{xi} that $\xi_{\theta}\ne 0,$ $\forall \, 0<\theta<1.$

 By the homotopy invariance of the Leray-Schauder degree, we have
 \begin{align}\label{1.36}
 \deg_{H^{\sigma }}(X_{1},\widehat{\Sigma}_{\tau}, 0)=
 \deg_{H^{\sigma}}(X_{0},\widehat{\Sigma}_{\tau},0).
 \end{align}
From \eqref{xi} and \eqref{eta}, we can obtain, for
$u=\sum_{i=1}^{k}\alpha_{i}\delta_{P_{i},t_{i}}+v\in \widehat{\Sigma}_{\tau},$
$$X_{0}(u)=\xi_{0}(u)+\eta_{0}(u),$$
where $\xi_{0}\in E_{P,t}, \eta_{0}\in \mathrm{span}\{\delta_{P_{i},t_i},\frac{\partial \delta_{P_{i},t_{i}}}{\partial t_{i}},\frac{\partial \delta_{P_i,t_{i}}}{\partial P_{i}}\}$ satisfy
\begin{align}
\begin{aligned}\label{x0}
\langle\xi_{0},\varphi\rangle
=&
\langle v,\varphi\rangle,
\\
\langle \eta_{0}, \delta_{P_{i},t_i}  \rangle
=&
-2\pi |\mathbb{S}^3|
(
 \alpha_{i}-K(\overline{P}_{i})^{-1}),\\
 \Big\langle \eta_{0}, \frac{\partial \delta_{P_i,t_i}}{\partial t_i}  \Big\rangle
=&
\frac{\Gamma_{3}}{K(\overline{P}_{i})^2}\frac{\tau}{t_{i}}
+\frac{\Gamma_{4}\Delta_{g_0}K(\overline{P}_{i})}{K(\overline{P}_i)^{3}}\frac{1}{t_i^3}
+\sum_{j\ne i}\frac{\Gamma_{5}G_{\overline{P}_{i}}(\overline{P}_{j})}{K(\overline{P}_i)K(\overline{P}_j)}
\frac{1}{t_i^2t_j}
,\\
\Big\langle
\eta_{0}, \frac{\partial \delta_{P_{i}, t_{i}}}{\partial P_{i}}
\Big\rangle
=&-\nabla_{g_{0}}K(P_{i}).
\end{aligned}
\end{align}
 Recalling the definition of $M(\overline{P}_{1},\cdots,\overline{P}_{k}),$
 which is simply written as $(M_{ij}).$  By \eqref{2.3}, we can easily get
 $$X_{0}(u)=0
\quad \hbox{on}\, \widehat{\Sigma}_{\tau},$$
if and only if
\begin{align}
\begin{aligned}\label{2.3}
&\alpha_{i}=K(\overline{P}_{i})^{-1}, \quad P_{i}=\overline{P}_{i},\quad v=0,\\
&\frac{4}{K(P_{i})^2}\frac{\tau}{t_{i}}-\Big(M_{ii}\frac{1}{t_{i}^3}
+\sum_{j=1}^{k}M_{ij}\frac{1}{t_{i}^{2}t_{j}}\Big)=0.
\end{aligned}
\end{align}

For any $(s_1,\cdots, s_{k})\in \mathbb{R}^{k},$ $s_{i}>0,$ $i=1,\cdots, k,$ we define
\begin{align*}
&F(s_{1},\cdots,s_{k}):=-
\sum_{j=1}^{k}\Big(\frac{4\tau}{K(\overline{P}_{j})^{2}}\log s_{j}
\Big)
+\frac{1}{2}\sum_{i=1}^{k}\Big(M_{ii}s_{i}^2
+\sum_{j=1}^{k}M_{ij}s_{i}s_{j}\Big),
\end{align*}
and for  $t_{i}=s_{i}^{-1},$
$$
\widehat{F}(t_1,\cdots,t_{k}):=F(s_{1},\cdots,s_{k}).
$$
The derivative with respect to $t_{i}$ is
\begin{align*}
\frac{\partial \widehat{F}}{\partial t_{i}}
(t_{1},\cdots,t_{k})=\frac{4}{K(\overline{P}_{i})^2}\frac{\tau}{t_{i}}-\Big(M_{ii}\frac{1}{t_{i}^3}
+\sum_{j=1}^{k}M_{ij}\frac{1}{t_{i}^{2}t_{j}}\Big),
\end{align*}
combining this and \eqref{x0}, we have
\begin{align*}
\Big\langle
\eta_{0}, \frac{\partial \delta_{P_i,t_i}}{\partial t_{i}}
\Big\rangle=\frac{\pi|\mathbb{S}^2|}{3}\frac{\partial \widehat{F}}{\partial t_{i}}(t_{1},\cdots, t_{k}).
\end{align*}

It is obvious that $\nabla\widehat{F}(t_{1},\cdots, t_{k})=0$ if and only if $\nabla F(s_{1},\cdots,s_{k})=0$.
A trivial verification shows that $F(s_{1},\cdots,s_{k})$ is a strictly convex function, and having a unique critical
point in the first quadrant.
It follows that $\widehat{F}(t_1,\cdots,t_k)$ has unique critical point in the first quadrant with Morse index
zero. Hence $X_{0}$ has precisely one nondegenerate zero in $\widehat{\Sigma}_{\tau}$.
Furthermore, by \eqref{2.3} we can easily obtain
  \begin{align}\label{1.35}
  \deg_{H^{\sigma}}(X_{0},\widehat{\Sigma}_{\tau},0)
  =(-1)^{k+\sum_{i=1}^{k}i(\overline{P}_{i})}.
  \end{align}
Combining \eqref{1.35} and \eqref{1.36},  we complete the
proof of Theorem \ref{thm3}.
\end{proof}

Recall the  definition of $\mathscr{O}_{R}$ in \eqref{1.82}. For $\delta>0$ suitably small, define
\begin{align}\label{1.85}
\mathscr{O}_{R,\delta}:=\{u\in H^{\sigma}(\mathbb{S}^{3}):
\inf_{\omega\in \mathscr{O}_{R}}\|u-\omega\|_{\sigma}<\delta\}.
\end{align}

\begin{proposition}\label{prop4}
Let $K\in \mathscr{A}$  be a Morse function and
$0<\tau_{0}\leq \tau \leq4/(n-2\sigma)-\tau_{0}.$
Then there exists some constants $C_{0}>0,$ $\delta_{0}>0$
depending only on $\tau_{0},$ $\min_{\mathbb{S}^3} K,$
and the modulo of the continuity of $K,$
such that
\begin{align}\label{0.21}
\{ u\in H^{\sigma}(\mathbb{S}^3):u>0~~  a.e.,\,I'_{\tau}(u)=0
\}\subset \mathscr{O}_{C_{0}, \delta_{0}}.
\end{align}
 Furthermore, we have $I_{\tau}'(u)\ne 0$ on $\partial \mathscr{O}_{C_0,\delta_{0}}$ and
\begin{align}\label{1.37}
\deg_{H^{\sigma} }\left(  u-P_{\sigma}^{-1}(K|u|^{1-\tau}u),
\mathscr{O}_{C_{0},\delta_{0}}, 0\right)=-1.
\end{align}
\end{proposition}
\begin{proof}
From  Proposition \ref{prop3},  we know that for $\tau>0$ small there exists some
   suitable value of $\nu_{0}, A, R$ such that
  $u$ satisfying $u\in H^{\sigma}(\mathbb{S}^3),
 u>0, a.e., I'_{\tau}(u)=0$ are
 either in $\mathscr{O}_{R}$ or in
 some $\Sigma_{\tau}(q^{(1)},\cdots, q^{(k)}).$
 Combining \eqref{stau}, \eqref{1.12},  \eqref{a.2}, and \eqref{p1p1},
    we conclude that there exists some
positive constants
$C_{0}$ and $\delta_{0}$ such that
  \eqref{0.21} holds.

For $K^{*}(x)=x^{(4)}+2,$ $x=(x^{(1)},x^{(2)},x^{(3)},x^{(4)})
\in\mathbb{S}^3\subset\mathbb{R}^{4}$ and $t\in (0,1),$ we consider $K_{t}=tK+(1-t)K^{*}.$
By the homotopy invariance of the Leray-Schauder degree, we only need to establish
\eqref{1.37} for $K^*$ and $\tau$ very small.
It is easy to see that $K^{*}\in \mathscr{A}$ is a Morse function. The proof of \eqref{1.37} is straightforward by  the Kazdan-Warner condition and Theorem \ref{thm3}.
\end{proof}

\subsection{The proof of the Theorems \ref{thm4} and \ref{thm2}}
Using Theorem \ref{thm3} and Proposition \ref{prop4}, we next prove Theorem \ref{thm4}.
\begin{proof}[Proof of Theorem \ref{thm4}]
Using Theorem \ref{thm2.1} and  the homotopy invariance of the Leray-Schauder degree,
for $\tau>0$ sufficiently small,
 we obtain that there exist a constant $R$ such that,
 \begin{align}\label{2.4}
 \deg_{C^{2,\alpha}}(u-P_{\sigma}^{-1}(Ku^{2}), \mathscr{O}_{R}, 0)
 =
\deg_{C^{2,\alpha}}(u-P_{\sigma}^{-1}(K|u|^{1-\tau}u), \mathscr{O}_{R}, 0).
 \end{align}

For $C_{0}\gg R,$ $0<\delta_{1}\ll \delta_{0},$
and  $\tau_{0}$ be given by  Proposition \ref{prop4}.
Using \eqref{1.37}, Proposition \ref{prop3}, \eqref{1.75}, and
the excision property of the degree, we have
\begin{align}\label{2.5}
\deg_{H^{\sigma}}(u-P_{\sigma}^{-1}(K|u|^{1-\tau}u), \mathscr{O}_{R,\delta_{1}}, 0)
=\mathrm{Index} (K).
\end{align}
As in the proof of Proposition \ref{prop4}, one can check that
there  are no critical points of $I_{\tau}$ in $\overline{\mathscr{O}_{R,\delta_{1}}}\backslash \mathscr{O}_{R}.$
Using the same proof idea as Li \cite[Theorem B.2]{LYYJ}, we can easily get
\begin{align}\label{2.6}
\deg_{C^{2,\alpha}}(u-P_{\sigma}^{-1}K(|u|^{1-\tau}u),\mathscr{O}_{R},0)
=\deg_{H^{\sigma}}(u-P_{\sigma}^{-1}K(|u|^{1-\tau}u),\mathscr{O}_{R,\delta_{1}},0).
\end{align}
It follows from \eqref{2.4}--\eqref{2.6} that for $R>C,$ \eqref{1.38} is proved.
Theorem \ref{thm4} follows from the above.
\end{proof}

Using the theory of linear algebra, we give the proof of  Corallary \ref{cor1}.
\begin{proof}[Proof of the Corollary \ref{cor1}]
If $\sharp \mathscr{K}^{-}=1,$ from the proof of Theorem \ref{thm4},
we can easily obtain the conclusion.
If for any distinct $P, Q\in \mathscr{K}^{-},$ $\Delta_{g_{0}} K(P)
\Delta_{g_{0}}K(Q)< 9K(P)K(Q),$  we claim that
there is no integer $k\geq 2$ such that $q^{(1)},\cdots, q^{(k)}\in
 \mathscr{K}^{-}$, $\mu(M(q^{(1)},\cdots, q^{(k)}))>0.$

 In fact, for any distinct $q^{(1)},\cdots,
 q^{(k)}\in \mathscr{K}^{-},$ $k\geq 2,$
 $\mu(M(q^{1},\cdots,q^{(k)}))>0$ if and only if
 $M(q^{(1)},\cdots,q^{(k)})$ is a positive definite matrix.
 By \eqref{1.97}, we have the fact that 2-order principle minor determinant  strictly less than zero.
Therefore, we proved the claim.
Obviously, Corollary \ref{cor1} follows from  the claim.
\end{proof}

We next prove Theorem \ref{thm2}.

\begin{proof}[Proof of the Theorem \ref{thm2}]
Since  the Morse functions in $C^{2}(\mathbb{S}^{3})^{+}\backslash\mathscr{A}=\partial \mathscr{A}$ are dense in $\partial \mathscr{A},$ without loss of generality we consider
the case that $K\in \partial\mathscr{A}$ is a Morse function.
 First recall the definition  of $\mathscr{K}$ and $\mathscr{K}^{+},$
 we can assume here  $\mathscr{K}\backslash\mathscr{K}^{+}=
 \{ q^{(1)}, \cdots, q^{(m)}\},$ $m\in \mathbb{N}_{+}.$
 From the definition of $\mathscr{A}$
 and $K\in \partial \mathscr{A},$ we know that
there exists $1\leq i_{1}<\cdots<i_{k}\leq m,$ $k\geq 1,$
 such that
 \begin{align}\label{2.9}
 \mu(M( q^{(i_1)}, \cdots, q^{(i_{k})}))=0.
 \end{align}
By perturbing the function $K$ near its some critical
points to change the Hessian matrix of $K$ at these points,
we obtain a sequence of ${K_{\ell}}$ satisfying:
 $K_{\ell}\rightarrow K$ in $C^{2}(\mathbb{S}^{3})^{+}$ as $\ell\to\infty;$
 ${K_{\ell}}$ are identically the same as $K$ except in some small
 balls and have the same critical points with the
same Morse index; there is only one such $(i_{1},\cdots, i_{k})$
such that \eqref{2.9} is true for any $\ell.$
 Refer to the perturbation method as in the proof of
  Li \cite[Theorem 0.8]{LYY} for more details.
  Using the same $C^{2}$ perturbation method for $K_{\ell},$  we can obtain a smooth, one-parameter
  family of Morse functions $\{K_{\ell,t}\}$ $(-1\leq t\leq 1)$ with the following
   properties:
\begin{itemize}
  \item [(a)]$K_{\ell,t}$ $(-1\leq t\leq 1)$ are identically the same as $K_{\ell}$ except in some small
 balls around $q^{(i_1)},\cdots, q^{(i_k)}$ and $K_{0}=K_{\ell}.$
 $K_{\ell,t}$  have the same critical points with the same Morse index  for any $-1\leq t\leq 1.$
  \item [(b)] $\mu(M_{\ell,t}( q^{(j_1)},\cdots, q^{(j_s)}))$ have the same sign for $-1<t<1$
for any $1\leq j_{1}<\cdots<j_{s}\leq m,$ $(j_{1},\cdots, j_{s})\ne (i_1,\cdots, i_{k}).$
$\mu(M_{\ell,t}( q^{(i_1)},\cdots, q^{(i_{k})}))<0$ for $-1<t<0,$
and $\mu(M_{\ell, t}(q^{(i_{1})}, \cdots, q^{(i_{k})})) >0$ for $0<t<1.$
\end{itemize}
It is easily seen that  $K_{\ell,t}\in \mathscr{A}$ when $t\ne 0.$  From the definition of
$\mathrm{Index},$  we have
\begin{align*}
\mathrm{Index}(K_1)
=\mathrm{Index}(K_{-1})+(-1)^{k-1+\sum_{j=1}^{k}i(q^{(i_{j})})},
\end{align*}
thus, $\mathrm{Index}(K_{1})\ne \mathrm{Index}(K_{-1}).$ By the homotopy invariance of
the Leray-Schauder degree, there exists $t_i\rightarrow 0$ and $v_{\ell,i}\in \mathscr{M}_{K_{\ell,t_i}},$
such that
\begin{align*}
\lim_{i\rightarrow \infty} \|v_{\ell,i}\|_{C^{2, \alpha}(\mathbb{S}^3)}
=\infty\quad \hbox{or}\quad \lim_{i\rightarrow \infty} (\min_{\mathbb{S}^3} v_{\ell,i})=0.
\end{align*}
Combining the Harnack inequality in \cite[Lemma 4.3]{JLX} and
Schauder estimates in \cite[Theorem 2.11]{JLX}, we deduce that \eqref{1.72}
holds. It follows from Theorem \ref{thm2.1},  $K_{\ell,t}\in \mathscr{A}$$(t\ne 0)$
and Theorem \ref{thm1} that $K_{\ell,t_{i}}\to K_{\ell}$ and $\{v_{\ell,i}\}$ blows up exactly at $k$ points
$q^{(i_{1})},\cdots,q^{(i_{k})}.$

From the above,  we know that there exists a sequence of $K_{i}\to K$ in $C^{2}(\mathbb{S}^3),$
$v_{i}\in \mathscr{M}_{K_{i}}$ such that $\{v_{i}\}$ blows up at precisely the $k$ points
$q^{(i_{1})},\cdots,q^{(i_{k})}.$
We have thus proved Theorem \ref{thm2}.
\end{proof}

\begin{proof}[Proof of Theorem \ref{thm5}]
By using Theorem \ref{thm1} we can prove the Part (i) of Theorem \ref{thm5}.
The Part (ii) of Theorem \ref{thm5} is similar to the proof of Theorem \ref{thm2}, we omit it here.
\end{proof}

\appendix
\section{Appendix}\label{sec5}
In this appendix, we provide some elementary calculations which have been used in the proof
of Theorem \ref{thm4}.

\begin{lemma}\label{cplema1}
Let  $\alpha \geq 2,$ there exists a positive constant
$C$ depending only on $\alpha$ such that, for any $a \geq 0,$ $b \in \mathbb{R},$
$$
\Big|
| a+b|^{\alpha-1}(a+b)-a^{\alpha}-\alpha a^{\alpha-1} b-\frac{\alpha(\alpha-1)}{2} a^{\alpha-2} b^{2}
 \Big|
  \leq C\left(|b|^{\alpha}+a^{\gamma}|b|^{\alpha-\gamma}\right),
$$
where $\gamma=\max\{0,\alpha-3\}.$
\end{lemma}

\begin{lemma}\label{lem1}
Let $1<\bt<2,$ there exists a universal positive constant $C$ such that, for any $a>0,$ $b\in \mathbb{R},$
\begin{align*}
\left||a+b|^{\bt-1}(a+b)-a^{\bt}-\bt a^{\bt-1}b\right|\leq C|b|^{\bt}.
\end{align*}
\end{lemma}

\begin{lemma}\label{lem2}
Let $\bt>1$ and $k\in \mathbb{N}_{+},$ there exists a constant $C$,
such that for any $(a_{1},\cdots, a_{k})
\in \mathbb{R}^{k},$
$$
\Big|\big(\sum_{i=1}^{k} a_i\big)^{\bt}-\sum_{i=1}^{k} a_{i}^{\bt}
\Big| \leq C\sum_{i \neq j}|a_{i}|^{\bt-1} |a_{j}|.
$$
\end{lemma}

\begin{lemma}\label{lem3}
 Let $\varepsilon_{0}, \tau>0$ be suitably small and $A>0$ be suitably large.
 Let $A^{-1} \tau^{-1 / 2}<t_{1}, t_{2}<A \tau^{-1 / 2},$
 $P_{1},P_{2}\in \mathbb{S}^3,$ $|P_{1}-P_{2}|\geq \varepsilon_{0},$
 $\delta_{P_{i},t_{i}}$ be as in
 \eqref{delta} and $G_{P_1}(P_2)$ be as in \eqref{gf}
 $(|P_{1}-P_{2}|$ represents the distance between  two points $P_{1}$
 and $P_{2}$  after through a stereographic  projection$).$
Then, we have,
\begin{align}\label{a.2}
&\int_{\mathbb{S}^3}\delta_{P_1, t_1}^2\delta_{P_2,t_2}
=4\pi|\mathbb{S}^2|\frac{G_{P_1}(P_2)}{t_1t_2}
+O(\tau^{3/2}),
\end{align}
\begin{align}\label{a.3}
\int_{\mathbb{S}^3}\delta_{P_1,t_1}^{2-\tau}\delta_{P_2,t_2}=O(\tau),
\end{align}
\begin{align}\label{a.1}
\frac{\partial}{\partial t_{1}}
\int_{\mathbb{S}^3}\delta_{P_1,t_{1}}^2\delta_{P_2, t_{2}}=
-4\pi|\mathbb{S}^2|\frac{G_{P_1}(P_2)}{t_1^2t_2}+O(\tau^2),
\end{align}
\begin{align}\label{a.4}
\frac{\partial}{\partial t_{1}}
\int_{\mathbb{S}^3}\delta_{P_2, t_{2}}\delta_{P_1,t_{1}}^{2-\tau}
=\int_{\mathbb{S}^3}\delta_{P_2,t_2}\delta_{P_1,t_1}^{2}+O(\tau^{5/2}|\log \tau|),
\end{align}
\begin{align}\label{partp12}
\frac{\partial}{\partial t_1}
\int_{\mathbb{S}^3}
\delta_{P_2, t_2}^{2-\tau}\delta_{P_1, t_1}=
\frac{\partial}{\partial t_1}\int_{\mathbb{S}^3}
\delta_{P_2, t_2}^{2}\delta_{P_1, t_1}+O(\tau^{5/2}|\log\tau|),
\end{align}
\begin{align}
\int_{\mathbb{S}^3}|P-P_1|^2\delta_{P_1,t_1}^{3-\tau}
=\frac{1}{t_{1}^2}\frac{3\pi}{2}|\mathbb{S}^2|
+O(\tau^{2}| \log \tau|),
\end{align}
\begin{align}\label{part3-tau}
\frac{\partial}{\partial t_1}\int_{\mathbb{S}^3}
\delta^{3-\tau}_{P_1,t_1}=
-\frac{\tau}{t_1}\frac{\pi}{2}|\mathbb{S}^2|
+O(\tau^{5/2}|\log \tau|),
\end{align}
\begin{align}\label{a.5}
\frac{\partial}{\partial t_1}
\int_{\mathbb{S}^3}
|P-P_1|^2\delta_{P_1,t_1}^{3-\tau}=
- \frac{3\pi}{t_{1}^3}|\mathbb{S}^2|+O(\tau^{5/2}|\log \tau|).
\end{align}
\end{lemma}

\begin{lemma}
Under the hypotheses of Lemma \ref{lem3}, in addition that $\Gamma_{1},\Gamma_{2}$ are
 positive constants independent of $\tau.$
Then, we have,
\begin{align}\label{p1p1}
\langle \delta_{P_{1}, t_{1}}, \delta_{P_{1}, t_{1}}\rangle=
|\mathbb{S}^3|,\quad
\Big\langle \frac{\partial\delta_{P_1,t_1}}
{\partial P_1^{(\ell)}},
\frac{\partial{\delta_{P_1,t_1}}}
{\partial P_1^{(\ell)}}
\Big\rangle=\Gamma_{1}t_1^2,
\end{align}
\begin{align}\label{part11}
\Big\langle
\frac{\partial\delta_{P_1, t_1}}{\partial t_{1}},
\frac{\partial\delta_{P_1, t_1}}{\partial t_{1}}
\Big\rangle=\Gamma_2 t_1^{-2},
\end{align}
\begin{align}\label{a.10}
\langle \delta_{P_{1}, t_{1}}, \delta_{P_{2}, t_{2}}\rangle
=O(\tau),
\end{align}
\begin{align}\label{1-tau}
\|\delta_{P_1,t_1}^{1-\tau}\delta_{P_2, t_2}\|_{L^{3/2}(\mathbb{S}^3)}=O(\tau|\log \tau|),
\end{align}
\begin{align}\label{a.11}
\Big\|
\delta_{P_2,t_2}^{1-\tau}
\frac{\partial \delta_{P_1,t_1}}{\partial t_1}\Big\|_{L^{3/2}(\mathbb{S}^3)}
= O(\tau^{3/2}|\log \tau|),
\end{align}
\begin{align}\label{12par1}
\Big\|\delta_{P_1,t_1}^{1-\tau}\delta_{P_{2}, t_{2}}\frac{\partial \delta_{P_1, t_1}}{\partial t_1}\Big\|_{L^1(\mathbb{S}^3)}=O(\tau^{3/2}|\log \tau|),
\end{align}
\begin{align}
\begin{aligned}\label{a.9}
&\|\delta_{P_1,t_1}^{2-\tau}-\delta_{P_1, t_1}^2 \|_{L^{3/2}(\mathbb{S}^3)}=
O(\tau |\log \tau|),\\\\
&\|\delta_{P_1,t_1}^{1-\tau}-\delta_{P_1, t_1} \|_{L^{3}(\mathbb{S}^3)}
=O(\tau |\log \tau|),
\end{aligned}
\end{align}
\begin{align}\label{a.7}
\|\delta_{P_1,t_1}^{3-\tau}-\delta_{P_1,t_1}^3\|_{L^{1}(\mathbb{S}^3)}
=O(\tau|\log \tau|),
\end{align}
\begin{align}
\begin{aligned}\label{p-p1}
&\big\||P-P_1|\delta_{P_1, t_1}^{2}\big\|_{L^{3/2}(\mathbb{S}^3)}=O( \tau^{1/2}),\\
&
\big\||P-P_1|^2 \delta^2_{P_{1}, t_1}\big\|_{L^{3/2}(\mathbb{S}^3)}
=O(\tau),
\end{aligned}
\end{align}
\begin{align}\label{a.8}
\int_{\mathbb{S}^3}|P-P_1|^3\delta_{P_1,t_1}
\frac{\partial \delta_{P_1,t_1}}{\partial t_1}=O(\tau^2).
\end{align}
\end{lemma}

\begin{lemma}
Let $\varepsilon_0, \tau, A$ be as in Lemma \ref{lem3}, $P_{1},P_2,P_3 \in \mathbb{S}^3$
satisfy $|P_i-P_j|\geq \varepsilon_0,$  $i\neq j,$
 and $A^{-1}\tau^{-1/2}<t_1,t_2,t_3\leq A\tau^{-1/2}.$ Then,  we have,
\begin{align}\label{a.12}
\Big|\frac{\partial }{\partial P_{1}}\int_{\mathbb{S}^3}
\delta_{P_{1},t_{1}}^{3-\tau}\Big|=O(\tau^{1/2}),\quad
\Big\| \delta_{P_2,t_2}
\frac{\partial \delta_{P_1,t_1}}{\partial P_1} \Big\|_{L^{3/2}(\mathbb{S}^3)}=O(
\tau^{1/2}|\log\tau|),
\end{align}
\begin{align}\label{a.7}
\Big|\frac{\partial }{\partial P_1}
\int_{\mathbb{S}^3}\delta_{P_2, t_2}^{2-\tau}\delta_{P_1, t_1}\Big|=O(\tau^{1/2}),\quad
\Big|\frac{\partial }{\partial P_1}
\int_{\mathbb{S}^3}\delta_{P_2, t_2}^2\delta_{P_1, t_1}\Big|=O(\tau^{1/2}),
\end{align}
\begin{align}\label{a.6}
&\Big\|
\delta_{P_2,t_2}^{1-\tau}\delta_{P_{3}, t_{3}}
\frac{\partial \delta_{P_1,t_1}}{\partial P_1}
\Big\|_{L^1(\mathbb{S}^3)}
=O(\tau^{1/2}|\log\tau|).
\end{align}
\end{lemma}

\begin{lemma}\label{lem4}
In addition to the hypotheses of Lemma \ref{lem3}, we assume that $ K\in C^{1}(\mathbb{S}^3)$. Then
\begin{align}\label{partk-kpj}
\frac{\partial}{\partial t_1}\int_{\mathbb{S}^3}(K(P)-K(P_2))\delta_{P_2, t_2}^{2-\tau}
 \delta_{P_1, t_1}=
O(\tau^2).
\end{align}

\end{lemma}

\begin{lemma}
In addition to the hypotheses of Lemma \ref{lem4},
 we assume that $ v\in E_{P_1,t_1}$. Then
\begin{align}\label{a.13}
&\int_{\mathbb{S}^3}
(K(P)-K(P_1))\delta_{P_1,t_1}^{1-\tau}
\frac{\partial \delta_{P_1,t_1}}{\partial P_{1}}v
=O(\tau^{1/2}|\log \tau| \|v\|_{\sigma}).
\end{align}

\end{lemma}

\bibliographystyle{plain}

\def\cprime{$'$}

\bigskip

\noindent Y. Li,\,\,Z. Tang,\,\,N. Zhou

\medskip

\noindent Email: \textsf{yanli@mail.bnu.edu.cn}\\
Email: \textsf{tangzw@bnu.edu.cn}\\
Email: \textsf{nzhou@mail.bnu.edu.cn}

\end{document}